\newif\ifpictures
\picturestrue

\documentclass[12pt]{amsart}
\usepackage{amssymb}
\usepackage{amsmath}
\usepackage{amscd}
\usepackage[pdftex]{graphicx}
\usepackage{hyperref}
\usepackage{bbm} % for the black board 1
\usepackage[textsize=tiny]{todonotes}

\numberwithin{equation}{section}
\newtheorem{thm}{Theorem}

\newtheorem{lemma}[thm]{Lemma}
\newtheorem{cor}[thm]{Corollary}

\theoremstyle{definition}

\newtheorem{example}[thm]{Example}
\newtheorem{remark1}[thm]{Remark}
\newtheorem{openproblem1}[thm]{Open problem}
\newtheorem{definition}[thm]{Definition}

\newenvironment{rem}{\begin{remark1}\rm}{\end{remark1}}

\numberwithin{thm}{section}

\newcounter{FNC}[page]
\def\newfootnote#1{{\addtocounter{FNC}{2}$^\fnsymbol{FNC}$%
     \let\thefootnote\relax\footnotetext{$^\fnsymbol{FNC}$#1}}}

\newcommand{\N}{\mathbb{N}}

\newcommand{\R}{\mathbb{R}}

\newcommand{\cX}{\mathcal{X}}
\newcommand{\cY}{\mathcal{Y}}

\DeclareMathOperator{\diag}{diag}

\DeclareMathOperator{\tr}{tr}

\DeclareMathOperator{\myspan}{span}

\DeclareMathOperator{\optstrat}{\mathcal{O}}

\newcommand{\sym}{\mathcal{S}}
\newcommand{\psd}{\mathcal{S}^{+}}

\title[Semidefinite games]{Semidefinite games}

\author{Constantin Ickstadt}

\author{Thorsten Theobald}

\address{Constantin Ickstadt, Thorsten Theobald:
        Goethe-Universit\"at, FB 12 -- Institut f\"ur Mathematik,
        Postfach 11 19 32, 60054 Frankfurt am Main, Germany}

\author{Elias Tsigaridas}

\address{Elias Tsigaridas: Sorbonne Universit\'e, Paris University, CNRS, and Inria Paris.
IMJ-PRG,  4 place Jussieu,
75252 Paris Cedex 05, France}

\date{\today}

\begin{document}

\begin{abstract}
We introduce and study the class of semidefinite games, which generalizes
bimatrix games and finite $N$-person games, 
by replacing the simplex of the mixed strategies for 
each player by a slice of the positive semidefinite cone
in the space of real symmetric matrices.

For semidefinite two-player zero-sum games, we show that the 
optimal strategies can be computed by semidefinite programming. 
Furthermore, we show that two-player
semidefinite zero-sum games are almost equivalent to semidefinite programming,
generalizing Dantzig's result on the almost equivalence of bimatrix
games and linear programming.

For general two-player semidefinite games, we prove a spectrahedral
characterization of the Nash equilibria. Moreover, we give constructions
of semidefinite games with many Nash equilibria. In particular,
we give a construction of semidefinite
games whose number of connected components
of Nash equilibria exceeds the long standing
best known construction for many Nash
equilibria in bimatrix games, which was presented 
by von Stengel in 1999.
\end{abstract}

\maketitle

\section{Introduction}
In the fundamental model of a bimatrix game in game theory,
the spaces of the mixed strategies are given by (two) simplices
\[
  \Delta_1 \ = \ \{ x \in \R^m \, : \, x \ge 0 \text{ and } \sum_{i=1}^m x_i = 1 \}, \quad
  \Delta_2 \ = \ \{ y \in \R^n \, : \, y \ge 0 \text{ and }
    \sum_{j=1}^n y_j = 1 \} .
\]
The payoffs of the two players, $p_A$ and $p_B$,
are given by two 
matrices $A, B \in \R^{m \times n}$, that is,
\[
  p_A(x,y) \ = \ \sum_{i,j} x_i A_{ij} y_j \; \text{ and } \; 
  p_B(x,y) \ = \ \sum_{i,j} x_i B_{ij} y_j.
\]
In the zero-sum case $B=-A$, optimal strategies do exist and can be 
characterized by linear programming. Moreover, by
Dantzig's classical result \cite{dantzig-1951-equivalence},
zero-sum matrix games and linear programming are almost equivalent,
see Adler \cite{adler-2013} and von Stengel \cite{von-stengel-2022}
for a detailed treatment of the situation
when Dantzig's reduction is not applicable. Bimatrix games can be
seen as a special case of broader classes of games (such as 
convex games, see \cite{ggf-2010},
separable games, see \cite{sop-2008}), 
which -- with increasing generality -- are less
accessible
from the combinatorial and computational viewpoint.

We introduce and study a natural semidefinite generalization
of bimatrix games (and of finite $N$-person games),
in which the strategy spaces are not simplices
but slices of the positive
semidefinite cone; that is
\begin{eqnarray*}
  \mathcal{X} & = & \{ X \in \sym_m \ : \  X \succeq 0 \text{ and } \tr(X) = 1\} \\
  \text {and } \, \mathcal{Y} & = & \{ Y \in \sym_n \ : \ Y \succeq 0 \text{ and } \tr(Y) = 1\} \, ,
\end{eqnarray*}
where $\sym_m$ denotes the set of real symmetric $m \times m$-matrices,
``$\succeq 0$'' denotes the positive definiteness of a matrix and $\tr$
abbreviates the trace.
The payoff functions are
\[
  p_A(X, Y) \ = \ \sum_{i,j,k,l} X_{ij} A_{ijkl} Y_{kl} \; \text{ and } \;
  p_B(X, Y) \ = \ \sum_{i,j,k,l} X_{ij} B_{ijkl} Y_{kl},
\]
where $A$ and $B$ are tensors in the bisymmetric space
$\sym_m \times \sym_n$. That is, $A$ satisfies the symmetry relations
$A_{ijkl} = A_{jikl}$ and $A_{ijkl} = A_{ijlk}$;
analogous symmetry relations hold for $B$.
If $X$ and $Y$ are restricted to be diagonal matrices, the semidefinite
games specialize to bimatrix games. Similarly, if $A_{ijkl} = B_{ijkl} = 0$
whenever $i \neq j$ or $k \neq l$, then the off-diagonal entries of $X$
and $Y$ do not have an influence on the payoffs
and the game is a special case of a bimatrix game.
\smallskip

The motivation for the model of semidefinite games 
comes from several origins.
\smallskip

\noindent
(1) The Nash equilibria of bimatrix games are intrinsically connected
to the combinatorics of polyhedra. Prominently, 
von Stengel \cite{von-stengel-1999} used
this connection and cyclic polytopes to construct a family of 
$n \times n$-bimatrix games whose number of equilibria grows as 
$0.949 \cdot (1+\sqrt{2})^n/\sqrt{n}$, for $n \to \infty$. 
In particular, this number grows faster
than $2^n-1$, which was an earlier conjecture of Quint and 
Shubik \cite{QuiShy-conj-97} for an 
upper bound. As of today, it is still an open problem whether there
are bimatrix games with even more Nash equilibria than in von 
Stengel's construction.
Recently, for tropical bimatrix games, Allamigeon, Gaubert and Meunier
\cite{agm-tc-ne-20} showed that the upper bound of $2^n-1$ equilibria,
i.e., the Quint-Shubik bound, holds in the tropical setting.

In many subareas around optimization and geometry, the transition 
from linear-polyhe\-dral settings to semidefinite settings has turned out
to be fruitful and beneficial. In this transition, polyhedra
(the feasible sets of linear programs) are carried
over into the more general spectrahedra
(the feasible sets of semidefinite programs),
see, e.g.,~\cite{bpt-2013}. 
One of our main goals is to relate the Nash equilibria of
semidefinite games to the geometry and combinatorics of spectrahedra.
\smallskip

\noindent
(2) Various approaches to quantum games have been investigated,
which combine game theoretic models with features of quantum 
computation and quantum information theory
(see \cite{berta2016quantum,guo-2008,ksb-2018,landsburg2011nash,
meyer-1999,sikora2017linear}). Quantum
states are given by positive semidefinite Hermitian matrices with unit trace
(see, e.g., \cite{nielsen-chuang-2002} or 
\cite{psv-2018} for an optimization
 viewpoint). 
  An essential characteristic of our model 
 is the use of positive
 semidefinite real-symmetric matrices with unit trace
 (also known as \emph{spectraplex}) as mixed strategies.
From this perspective, we can consider the semidefinite games
 as a real-quantum generalization of 
bimatrix games and of finite $N$-player games.

Our class of games can also be seen as a subclass of the interactive
quantum games studied in \cite{GutWat-qgc-07}, see 
also \cite{BW-qaNE-21}.
These games involve two players and a referee and, possible many, interactions between them. The overall actions of each player, 
the so-called Choi representation, consist of a single Hermitian 
positive semidefinite matrix along with a finite number of
linear constraints
and for the zero-sum case they derive a minimax theorem over the complex numbers. 
We refer the reader to \cite{BW-qaNE-21} for further details and complexity results
and to \cite{jain2009parallel}
for an algorithm to compute the equilibrium in the one round 
zero-sum case. 
\smallskip

\noindent
(3) In recent times, the connection of games and the use of polynomials in
  optimization has received wide interest. Prominently, 
  Stein, Ozdaglar and Parrilo \cite{parrilo-2006,sop-2008,
  stein-2011-correlated} have developed sum of 
  squares-based optimization solvers for game theory.
  Laraki and Lasserre have developed hierarchical moment 
  relaxations \cite{laraki-lasserre-2012},
  see also Ahmadi and Zhang \cite{ahmadi-zhang-2021} 
  for semidefinite relaxations and the Lasserre hierarchy
  to approximate Nash equilibria in bimatrix games.
  Recently, Nie and Tang \cite{nie-tang-2020,nie-tang-2021-generalized-nash}
  have studied games with polynomial
  descriptions and convex generalized
  Nash equilibrium problems through polynomial optimization and
  moment-SOS relaxations.
  The semidefinite conditions correspond to a nice polynomial structure, 
  with underlying convexity. 
   
   In a different direction, the geometry of Nash equilibria in our class of
   games establishes novel connections and questions between 
   game theory and semialgebraic geometry. Here, recall that already
   the set of Nash equilibria of finite $N$-person games can be 
   as complicated as arbitrary semialgebraic 
   sets \cite{datta-2003}.
   See \cite{portakal-sturmfels-2022} for recent work on the geometry 
   of dependency  equilibria.
\smallskip

\noindent
{\bf Our contributions.}
1. We develop a framework for approaching semidefinite games
  through the duality theory of semidefinite programming.
  As a consequence, the optimal strategies in semidefinite zero-sum games
  can be computed by a semidefinite program.
  Moreover, the set of optimal strategies are spectrahedra
  (rather than only projections of spectrahedra).
  See Theorem~\ref{th:zero-sum-spectrahedron}.
\smallskip

\noindent
2. We generalize Dantzig's result on the almost equivalence of 
  zero-sum bimatrix games and linear programs to 
  the almost equivalence of semidefinite zero-sum
  games and semidefinite programs. See 
  Theorem~\ref{th:equivalence1}. For the special case of
  semidefinite programs with diagonal matrices, our result 
  recovers Dantzig's result.
\smallskip

\noindent
3. For general (i.e., not necessarily zero-sum) semidefinite games, 
  we prove a spectrahedral characterization of Nash equilibria.
  This characterization generalizes the polyhedral characterizations
  of Nash equilibria in bimatrix games.
  See Theorem~\ref{th:nash-charact1}.
\smallskip

\noindent
4. We give constructions of families of semidefinite games with many
  Nash equilibria. In particular, these constructions of games on the
  strategy space $\sym_n \times \sym_n$ have more connected
  components of Nash equilibria than the best known constructions
  of Nash equilibria in bimatrix games (due to von 
  Stengel \cite{von-stengel-1999}).
  See Example~\ref{ex:many-nash1}.
\smallskip

The paper is structured as follows. After collecting some notation in
Section~\ref{se:notation}, we introduce semidefinite
games in Section~\ref{se:semidefinite-games} and view them 
within the more general class of separable games.
Section~\ref{se:zero-sum}
deals with computing the optimal strategies in semidefinite
zero-sum games by semidefinite programming. Section~\ref{se:equivalence}
then proves the almost equivalence of zero-sum games and semidefinite
programs. For general semidefinite games,
Section~\ref{se:general-semidefinite-games} gives
a spectrahedral characterization of the Nash equilibria. In
Section~\ref{se:many-nash}, we present constructions with
many Nash equilibria. Section~\ref{se:outlook} concludes the
paper.

\section{Notation\label{se:notation}}
We denote by $\sym_n$  the set of real symmetric $n \times n$-matrices
and by $\sym_n^+$ the subset of matrices in $\sym_n$ which are positive
semidefinite. Further, denote by
$\langle \cdot, \cdot \rangle$ the Frobenius scalar product,
$\langle A, B \rangle := \sum_{i,j} a_{ij} b_ {ij}$.
$I_n$ denotes the identity matrix.

An optimization problem of the form
\begin{equation}
  \label{eq:sdp-form1}
  \inf_{X \in \sym_n} \left\{ \langle C, X \rangle \, : \ \langle A_i, X \rangle = b_i, \, 1 \le i \le m, \, X \succeq 0 \right\}
\end{equation}
with $A_1, \ldots, A_m \in \sym_n$, $C \in \sym_n$ and $b \in \R^m$
is called \emph{semidefinite program (SDP) in primal normal form}, 
and a problem of the form 
\begin{equation}
  \label{eq:sdp-form2}
  \sup_{Z \in \sym_n, \, y \in \R^m} \{ b^T y\, : \, \sum_{i=1}^m y_i A_i + Z = C, \, Z \succeq 0\}
\end{equation}
is called an \emph{SDP in dual normal form.} We will make frequent use
of the following duality results of semidefinite programming, 
see, e.g., \cite{vandenberghe-boyd-survey}.

\begin{thm} (a) \emph{(Weak duality.)}
Let $X$ and $(Z,y)$ be feasible points for~\eqref{eq:sdp-form1}
and~\eqref{eq:sdp-form2}. Then
$
  \langle C,X \rangle - b^T y \ \ge 0 \, .
$
\smallskip

(b) \emph{(Strong duality.)}
If both \eqref{eq:sdp-form1} and~\eqref{eq:sdp-form2} are
strictly feasible with finite optimal values, then the optimal values
coincide and they are attained in both problems.
\end{thm}

A convex set $C \subset \R^k$ is called a \emph{spectrahedron} if it can be written
in the form
\begin{equation}
  \label{eq:spectrahedron1}
  C = \Big\{ x \in \R^k \ : \ A_0 + \sum_{i=1}^k x_i A_i \succeq 0 \Big\},
\end{equation}
with $A_0, \ldots,A_k \in \sym_n$ for some $n \in \N$. Any representation
of $C$ of the form~\eqref{eq:spectrahedron1} is called an
\emph{LMI (Linear Matrix Inequality) representation} of $C$.

A spectrahedron in $\R^k$ can also be described as the intersection of the
cone $\sym_n^+$ with an affine subspace
$
  U = A_0 + L ,
$
where $A_0 \in \sym_n$ and 
$L$ is a linear subspace of $\sym_n$ of
dimension $k$, say, given as $L = \myspan\{A_1, \ldots, A_k\}$ (see, e.g., \cite{rostalski-sturmfels-2010}, \cite[Chapter~5]{bpt-2013}).
The sets of the form
\begin{equation}
  \label{eq:spectrahedral-shadow}
  C = \Big\{ x \in \R^k \ : \ \exists y \in \R^l \; \:
      A_0 + \sum_{i=1}^k x_i A_i + \sum_{j=1}^l y_j B_j \succeq 0 \Big\},
\end{equation}
with symmetric matrices $A_i, B_j$,
are called \emph{spectrahedral shadows} 
(see \cite{scheiderer-spectrahedral-shadows}). Any representation
of the form~\eqref{eq:spectrahedral-shadow} is called
a \emph{semidefinite representation of $C$.}

\section{Semidefinite games\label{se:semidefinite-games}}

\subsection{Two-player and \texorpdfstring{$N$}{N}-player 
  semidefinite games}
  % in formulas in sections, texorpdfstring has to be used to avoid a warning
Most of our work is concerned with two-player semidefinite games.
For simplicity, we work over the real numbers, while many considerations
can also be carried over to the complex numbers.
Let $m, n \ge 1$ and the strategy spaces $\cX$ and $\cY$ are
\begin{eqnarray*}
  \mathcal{X} & = & \{ X \in \sym_m \ : \  X \succeq 0 \text{ and } \tr(X) = 1\} \\
  \text {and } \mathcal{Y} & = & \{ Y \in \sym_n \ : \ Y \succeq 0 \text{ and } \tr(Y) = 1\} \, .
\end{eqnarray*}

To formulate the payoffs, it is convenient to denote by
$(A_{\cdot\cdot kl})_{1 \le k,l \le n}$ the symmetric 
$n \times n$-matrix which results from a fourth-order tensor
$A$ by fixing the third index to $k$ and the fourth index to $l$.
Such two-dimensional sections of a tensor are also 
called \emph{slices}.
The payoff functions are
\begin{eqnarray*}
  p_A(X, Y) \ = \ \sum_{i,j,k,l} X_{ij} A_{ijkl} Y_{kl} 
       \  = \ \langle ( \langle X, A_{\cdot \cdot kl} \rangle)_{1 \le k,l \le n}, Y
            \rangle \\
  \text{ and } \;
  p_B(X, Y) \ = \ \sum_{i,j,k,l} X_{ij} B_{ijkl} Y_{kl} 
        \ = \ \langle ( \langle X, B_{\cdot \cdot kl} \rangle)_{1 \le k,l \le n}, Y
            \rangle ,
\end{eqnarray*}
where $A$ and $B$ are tensors in the bisymmetric space
$\sym_m \times \sym_n$. That is, $A$ satisfies the symmetry relations
$A_{ijkl} = A_{jikl}$ and $A_{ijkl} = A_{ijlk}$
and analogous symmetry relations hold for $B$.
If $A = -B$, then the game is called a \emph{semidefinite zero-sum game}.

For the $N$-player version, 
with strategy spaces
\[
  \mathcal{X}^{(i)} = \{X \in \sym_{m_i} \, : \, X \succeq 0
  \text{ and } \tr(X) = 1\},
  \quad \text{for } 1 \leq i \leq N,
\]
let $A^{(1)}, \ldots, A^{(N)} \in \sym_{m_1} \times \cdots \times 
  \sym_{m_N}$.
If $X=(X^{(1)}, \ldots, X^{(N)})$, then the
payoff function for the $k$-th player is
\begin{eqnarray*}
  p_k(X^{(1)}, \ldots, X^{(N)}) 
      & = &
            \sum_{i_1,j_1=1}^{m_1} \cdots \sum_{i_N,j_N=1}^{m_N}
            A^{(k)}_{(i_1,j_1), \ldots, (i_N,j_N)} X^{(1)}_{(i_1,j_1)} \cdots 
            X^{(N)}_{(i_N,j_N)} \, .
\end{eqnarray*}

\subsection{Separable games}

Stein, Ozdaglar and Parrilo \cite{sop-2008} have introduced the 
class of separable games.
An $N$-player \emph{separable game} consists of pure
strategy sets $C_1, \ldots, C_N$,
which are non-empty
compact metric spaces, and the payoff 
functions $p_k : C \to \R$. The latter are of the form
\[
  p_k(s) \ = \ \sum_{j_1=1}^{m_1} \cdots \sum_{j_N=1}^{m_N}
    a_k^{j_1 \cdots j_N} f_1^{j_1}(s_1) \cdots f_N^{j_N}(s_N),
\]
where $C := \prod_{k=1}^N C_k$,
$a_k^{j_1 \cdots j_N} \in \R$, the functions
$f_k^{j_{i}} : C_k \to \R$ are continuous,
and $i, k \in \{1, \ldots, N\}$.

Semidefinite games are special cases of separable games.
We can see this relation from two viewpoints.
From a first viewpoint, let $C_k$ be the matrices in 
$\sym_{m_k}^+$ with trace~1
and set
\begin{eqnarray*}
  f_t^{(r,s)}(X^{(t)}) & = & X^{(t)}_{rs}, \\
  a_k^{(i_1,j_1),\ldots,(i_N,j_N)} & = & (A^{(k)})_{(i_1,j_1),\ldots,(i_N,j_N)}.
\end{eqnarray*}
Then, the payoff functions become
\begin{eqnarray*}
    p_k(X^{(1)}, \ldots, X^{(N)}) & = & 
      \sum_{i_1,j_1=1}^{m_1} \cdots \sum_{i_N,j_N=1}^{m_N}
      a_k^{(i_1,j_1), \ldots, (i_N,j_N)} f_1^{(i_1,j_1)}(X^{(1)}) \cdots f_n^{(i_N,j_N)}(X^{(N)}) \\
      & = &
            \sum_{i_1,j_1=1}^{m_1} \cdots \sum_{i_N,j_N=1}^{m_N}
            A^{(k)}_{(i_1,j_1), \ldots, (i_N,j_N)} X^{(1)}_{(i_1,j_1)} \cdots 
            X^{(N)}_{(i_N,j_N)}.
\end{eqnarray*}
This yields the setup of semidefinite games as introduced before.

The set of mixed strategies $\Delta_k$ of the $k$-th player
is defined as the space of Borel probability measures $\sigma_k$
over $C_k$. 
A mixed strategy profile $\sigma$ is a \emph{Nash equilibrium}
if it satisfies
\[
  p_k(\tau_k, \sigma_{-k}) \le p_k(\sigma),
  \quad \text{for all $\tau_k \in \Delta_k$ and
  $k \in \{1, \ldots, N\}$},
\]
where $\sigma_{-k}$ denotes the mixed strategies of all players 
except player $k$.

In this setting, the relation of our model to the mixed strategies
of separable games does not yield any new insight,
since taking the Borel measures over the convex set
$C_k$ does not give new strategies.

There is a second viewpoint, which better captures the role of the
pure strategies. Since every point in the
positive semidefinite cone
is a convex combination of positive semidefinite rank-1 matrices,
we can also define the set of pure strategies $C_k$ as 
the set of matrices in $\sym_{m_k}^+$ which have trace~1 and rank~1.
Then, 
by a Carath\'{e}odory-type argument in \cite[Corollary 2.10]{sop-2008},
every separable game has a Nash equilibrium
in which player $k$ mixes among at most 
$\dim \sym_{m_k}+1  = \binom{m_k+1}{2}+1$ pure
strategies.
In contrast to finite $N$-player games, the decomposition of a mixed strategy
(such as the one in a Nash equilibrium) in terms of the pure strategies is not unique.
Example~\ref{ex:5nasheqb} will illustrate this.

\section{Semidefinite zero-sum games\label{se:zero-sum}}

In this section, we consider semidefinite zero-sum games.
The payoff tensors are given by 
$A$ and $B:=-A$. Hence, the second player wants to minimize the
payoff of the first player, $p_A$. By the classical minimax theorem for bilinear functions 
over compact convex sets
\cite{dks-1950,von-neumann-1945}
(see also \cite{dresher-karlin-1953}), 
optimal strategies exist in the zero-sum case. We show that the
sets of optimal strategies are spectrahedra and reveal the semialgebraic
geometry of semidefinite zero-sum games.

\begin{thm} \label{th:zero-sum-spectrahedron}
Let $G=(A,B)$ be a semidefinite zero-sum game. 
Then, the set of optimal strategies of each player
is the set of optimal solutions of a semidefinite program.
Moreover, each set of optimal strategies is a spectrahedron.
\end{thm}

The \emph{value} $V$ of the game is defined through the minimax relation
\[
  \max_{X \in \cX} \min_{Y \in \cY} 
    \sum_{i,j,k,l} X_{ij} A_{ijkl} Y_{kl} \ = \ V
    \ = \ \min_{Y \in \cY} \max_{X \in \cX}
    \sum_{i,j,k,l} X_{ij} A_{ijkl} Y_{kl}.
\]

The following lemma records that zero-sum
matrix games can be embedded into semidefinite
zero-sum games.

\begin{lemma}
\label{le:embed}
For a given zero-sum matrix game $G$ with payoff matrix 
$A = (a_{ij}) \in \R^{m \times n}$,
let $G'$ be the semidefinite zero-sum game on $\sym_m \times \sym_n$-matrices and payoff tensor
\[
\begin{array}{rcl@{\quad}l}
  A_{iikk} & = & a_{ik} & \text{for } 1 \le i \le m, \; 1 \le k \le n, \\
  A_{ijkl} & = & 0         & \text{for } i \neq j \text{ or } k \neq l.
\end{array}
\]
Then a pair $(x,y) \in \Delta_1 \times \Delta_2$ is a pair of optimal strategies for $G$
if and only if there exists a pair of optimal strategies $X \times Y$ for $G'$ with
\begin{equation}
  \label{eq:diagonal1}
  X_{ii} = x_{i}, \; 1 \le i \le m,  \; \text{ and } \;
  Y_{kk} = y_{k}, \; 1 \le k \le n.
\end{equation}
\end{lemma}

\begin{proof}
For any strategy pair $(X,Y) \in \cX \times \cY$ with \eqref{eq:diagonal1},
the payoff in $G'$ is
\[
  \sum_{i,j,k,l} X_{ij} A_{ijkl} Y_{kl} \ = \ \sum_{i,k} X_{ii} A_{iikk} Y_{kk}
  \ = \ \sum_{i,k} x_i a_{ik} y_k,
\] 
which coincides with the payoff in $G$ for the strategy pair $(x,y)$.
\end{proof}

As a consequence of Lemma~\ref{le:embed}, any oracle to solve
semidefinite zero-sum games can be used to solve zero-sum matrix 
games. Namely, construct the semidefinite zero-sum game 
$G'$ described in Lemma~\ref{le:embed} and let $X^* \in \sym_m^+$ 
and $Y^* \in \sym_n^+$ be the optimal strategies provided by the oracle.
Let $x^*$ and $y^*$ be the vectors of diagonal elements of
$X^*$ and $Y^*$. Since $X^*$ and $Y^*$ are positive semidefinite,
the vectors $x^*$ and $y^*$ are nonnegative and due to
$\tr(X^*) = \tr(Y^*) = 1$ we have 
$\sum_{i=1}^m x_i^*=\sum_{j=1}^n y_i^* = 1$.
Since $A_{ijkl} = 0$ for $i \neq j$ or $k \neq l$, the off-diagonal
elements in any strategy of the semidefinite game $G'$ do not 
matter for the payoffs. 
Hence, $x^*$ and $y^*$ are optimal strategies
for the zero-sum matrix game.

\begin{lemma}
\label{le:antisymm}
Let $G$ be a semidefinite zero-sum game on $\sym_n \times \sym_n$. 
If the payoff tensor satisfies 
\[
  A_{ijkl} = - A_{klij} \quad \text{for all } i,j,k,l,
\]
then $G$ has value~0.
\end{lemma}

In the proof, we employ a simple symmetry consideration.

\begin{proof}
Let $V$ denote the value of $G$. Then, there exists an $X \in \mathcal{X}$ such
that for all $Y \in \mathcal{Y}$, we have
$ \sum_{i,j,k,l} X_{ij} A_{ijkl} Y_{kl} \ \ge \ V.$
In particular, this implies 
\begin{equation}
  \label{eq:antisymm1}
  \sum_{i,j,k,l} X_{ij} A_{ijkl} X_{kl} \ \ge \ V.
\end{equation}
If we rearrange the order of the summation and use the precondition, then 
\begin{equation}
  \label{eq:antisymm2}
  \sum_{i,j,k,l} X_{ij} A_{klij} X_{kl} \ = \ \sum_{i,j,k,l} - X_{ij} A_{ijkl} X_{kl} \ \ge \ V.
\end{equation}
Adding~\eqref{eq:antisymm1} and~\eqref{eq:antisymm2} yields $V \le 0$.
Analogously, there exists some $Y \in \mathcal{Y}$ such that for all $X \in \mathcal{X}$, 
we have $ \sum_{i,j,k,l} X_{ij} A_{ijkl} Y_{kl} \ \le \ V.$ Arguing similarly as before, we
can deduce $V \ge 0$. Altogether, this gives $V = 0$.
\end{proof}

We characterize which a-priori-strategy player~1 will play if her
strategy will be revealed to player~2 (max-min-strategy). In the following
lemma, the symmetric $n \times n$-matrix $T$ plays the role of a
slack matrix.

\begin{lemma}
\label{le:apriori1}
Let $(A,B)$ be a semidefinite zero-sum game.
As an a-priori-strategy, player~1 plays an optimal solution 
$X$ of the SDP
\begin{equation}
  \label{eq:sdp-minmax1}
         \max_{X,T \succeq 0, \, v_1 \in \R} \big\{ v_1 \, : \, v_1 I_n + T = (\langle X, A_{\cdot\cdot kl} \rangle)_{1 \le k,l \le n} , \, \tr(X) = 1 \big\}.
\end{equation}
The optimal value of this optimization problem is attained.
\end{lemma}

\begin{proof}
For an a priori known strategy of player~1, player~2 will play a
best response, i.e., an optimal solution of the problem
\begin{eqnarray}
  & & \min_Y \{ p_A(X,Y) \ : \ \tr(Y) = 1, \, Y \succeq 0\} \nonumber \\
  & = &  \min_Y \{ \langle (\langle X, A_{\cdot\cdot kl} \rangle)_{1 \le k,l \le n}, Y \rangle  \ : \ \tr(Y) = 1, \, Y \succeq 0\} \label{eq:primal1} \, .
\end{eqnarray}
In what follows we see that the optimal value of the 
minimization problem is attained and that strong duality holds. 
As a-priori-strategy, player~1 uses an optimal solution of
\[
  \max_{X \succeq 0, \, \tr(X) = 1} \min_Y \{ \langle ( \langle X, A_{\cdot\cdot kl} \rangle)_{1 \le k,l \le n}, Y \rangle \ : \ \tr(Y) = 1, \, Y \succeq 0\} \, .
\]

We write the inner minimization
problem of the minmax problem in terms of
the dual SDP. This gives
\begin{eqnarray}
    & & \min_Y \{ \langle ( \langle X,A_{\cdot\cdot kl} \rangle)_{1 \le k,l \le n}, 
    Y \rangle \ : \ \tr(Y) = 1, \, Y \succeq 0\}, \label{eq:dual3} 
    \label{eq:dual4a}\\
     & = & \max_{T, \, v_1} \{ 1 \cdot v_1 \, : \, v_1 I_n + T = (\langle X,A_{\cdot\cdot kl}  \rangle)_{1 \le k,l \le n}, \, T \succeq 0, v_1 \in \R \} \, .
     \label{eq:dual4b}
\end{eqnarray}
Note that the scaled unit matrix $\frac{1}{n} I_n$
is a strictly feasible point for the
minimization problem~\eqref{eq:dual3}. If we choose a negative $v_1$
with sufficiently large absolute value, then the maximization 
problem~\eqref{eq:dual4b} has a strictly feasible point as well. Hence,
the duality theory for semidefinite programming implies that both the
minimization and the maximization problems attain the optimal value.
In connection with the outer 
maximization this
gives the semidefinite program~\eqref{eq:sdp-minmax1}.
\end{proof}

We remark that 
the expressions~\eqref{eq:dual4a} and~\eqref{eq:dual4b} can be 
interpreted as the smallest eigenvalue of the matrix 
$(\langle X,A_{\cdot\cdot kl}  \rangle)_{1 \le k,l \le n}$
(see, e.g., \cite{vandenberghe-boyd-survey}).
Further note that the SDP~\eqref{eq:sdp-minmax1}
is not quite in one of the normal forms, 
see Lemma~\ref{le:strat-duality} below.

Next, we characterize the a-priori-strategies of player~2 in terms of a
minimization problem to facilitate the duality reasoning. Similar to 
Lemma~\ref{le:apriori1}, the symmetric $n \times n$-matrix $S$
serves as a slack matrix.

\begin{lemma}
\label{le:apriori2}
Let $(A,B)$ be a semidefinite zero-sum game.
As an a-priori-strategy, player~2 plays an optimal solution $Y$ of the SDP
\begin{equation}
  \label{eq:sdp-minmax2}
          \min_{Y, S \succeq 0, \, u_1 \in \R} \{ - u_1 \, : \, u_1 I_n + S = (\langle -A_{ij\cdot\cdot}, Y \rangle)_{1 \le i,j \le n}, \, \tr(Y) = 1\} \, .
\end{equation}
The optimal value of this optimization problem is attained.
\end{lemma}

\begin{proof}
The proof is similar to Lemma~\ref{le:apriori1}.
As an a-priori-strategy, player~2 uses an optimal solution of
\[ 
  \min_{Y \succeq 0, \, \tr(Y) = 1} \max_X \{ \langle (\langle A_{ij\cdot\cdot}, Y\rangle)_{1 \le i,j \le n}, X \rangle \ : \ \tr(X) = 1, \, X \succeq 0\} \, .
\]
Since both the inner optimization problem and its dual are strictly
feasible, strong duality holds. The statement then follows from the
duality relation
\[
 \max_{X \succeq 0} \{ \langle (\langle A_{ij\cdot\cdot}, Y\rangle)_{i,j}, X \rangle \ : \ 
    \langle I_n, X \rangle = 1 \} 
     \ = \ \min_{S \succeq 0, \, u_1 \in \R} \{ - u_1 \, : \, u_1 I_n + S = (\langle -A_{ij\cdot\cdot}, Y \rangle)_{i,j} \}.
\]
\end{proof}

\begin{rem}

Similar to the case of zero-sum matrix games, if the coefficients of the
payoff tensor are chosen sufficiently generically, 
then the SDPs~\eqref{eq:sdp-minmax1} and~\eqref{eq:sdp-minmax2}
have a unique optimal solution and, as a consequence, each player has
a unique optimal strategy.
In contrast to the set of optimal strategies
of zero-sum matrix games, it is possible that the set of optimal strategies 
of a semidefinite game is non-polyhedral.
For example, already in the trivial semidefinite game with the zero matrix $A$, 
the value is 0 and the set of optimal strategies of
player 1 and player 2 are the full sets $\cX$ and $\cY$, respectively.
\end{rem}

Now we show that the sets of optimal strategies of the two players
can be regarded as the optimal solutions of a pair of dual SDPs.

\begin{lemma}
\label{le:strat-duality}
The SDPs~\eqref{eq:sdp-minmax1} and~\eqref{eq:sdp-minmax2} are dual
to each other.
\end{lemma}

\begin{proof}
We show that the dual of~\eqref{eq:sdp-minmax2} coincides 
with~\eqref{eq:sdp-minmax1}. Setting
\[
  Y' \ = \ \diag(Y,S,u_1^+, u_1^-),
\]
i.e., the block diagonal matrix with blocks $Y$, $S$, $u_1^+$
and $u_1^-$ (of size $\sym_n, \sym_n, 1, 1)$, 
the problem~\eqref{eq:sdp-minmax2} can be written as
\begin{equation}
\label{eq:sdp3}
\begin{array}{rcl}
  \multicolumn{3}{l}{\min \big\langle 
  \diag(0,0,-1,1), Y' \big\rangle} \\
  \delta_{ij} (u_1^+ - u_1^-) + s_{ij} + \langle A_{ij\cdot\cdot}, Y \rangle & = & 0
    \quad \text{for all } 1 \le i \le j \le n \, , \\
  y_{11} + \cdots + y_{nn} & = & 1 \, , \\
  Y' & \succeq & 0 \, .
\end{array}
\end{equation}

We claim that the dual of~\eqref{eq:sdp3} coincides 
with~\eqref{eq:sdp-minmax1}. Denote by $E_{ij}$ the matrix with $1$ in
row $i$ and column $j$ whenever $i=j$ and with $1/2$ in row $i$
and column $j$ as well as row $j$ and column $i$ otherwise.
The dual is
\begin{equation}
  \label{eq:dual2}
  \begin{array}{rcl}
  \multicolumn{3}{l}{\max_{S',W, w'} w'} \\
    \sum w_{ij} \diag( (A_{ij\cdot\cdot}),
        E_{ij},
        \delta_{ij},
        -\delta_{ij} )
    + w' \diag(I_n,0,0,0)
      + S' & = & \diag(0,0,-1,1), \\
  S' & \succeq & 0, \; W \in \sym_n, \, w' \in \R \, .
  \end{array}
\end{equation}
Observe that $\sum_{i,j} w_{ij} A_{ijkl} = \langle W, A_{\cdot\cdot kl}\rangle$.
The second block in the constraint matrices gives that $W$ is minus the second
block of $S'$, which describes a positive semidefiniteness condition on $-W$.
Then the equations involving $\delta_{ij}$ and $-\delta_{ij}$
ensure that $\sum_{i=1}^n w_{ii} = -1$; namely,
in~\eqref{eq:dual2}, each of 
the two corresponding equations contains a non-negative slack variable.
The combination of these equations shows that both of these slack variables
must be zero. Altogether, this gives
\begin{equation}
  \label{eq:dual7}
         \max_{S^*, -W \succeq 0, w' \in \R} \big\{ w' \, : \, w' I_n + S^* = (\langle -W, A_{\cdot\cdot kl} \rangle)_{1 \le k,l \le n} , \, 
  \tr(W) = -1 \big\}.
\end{equation}
By identifying $W$ with $-X$, we recognize this as the SDP in
~\eqref{eq:sdp-minmax1}
\end{proof}

Now we provide the proof of Theorem~\ref{th:zero-sum-spectrahedron}.

\begin{proof}
Player~1 can achieve at least the gain provided by the 
a-priori-strategy \eqref{eq:sdp-minmax1} from Lemma~\ref{le:apriori1}, and 
player~2 can bound her loss by the a-priori-strategy~\eqref{eq:sdp-minmax2} 
from Lemma~\ref{le:apriori2}. By Lemma~\ref{le:strat-duality}, both strategies
are dual to each other, so that their optimal values coincide with the value of
the game. In the coordinates $(X,T,v_1)$ and
$(Y,S,u_1)$, the feasible regions of those optimization problems are spectrahedra,
and the sets of optimal strategies are the sets of optimal solutions of the SDPs.
Intersecting the feasibility spectrahedron, say, for player~1, with
the hyperplane corresponding to the optimal value of the objective function
shows the first part of the 
theorem. Precisely, we obtain the sets
\begin{eqnarray*}
  & & \{X, T \succeq 0 \, : \, 
    V I_n + T = (\langle X, A_{\cdot\cdot kl} \rangle)_{1 \le k,l \le n} , \, \tr(X) = 1 \big\} \\
 & \text{ and } \; 
          & \{ Y, S \succeq 0 \, : \, - V I_n + S = (\langle -A_{ij\cdot\cdot}, Y \rangle)_{1 \le i,j \le n}, \, \tr(Y) = 1\} \, ,
\end{eqnarray*}  
where $V$ is the value of the game.
By projecting the spectrahedra for the first player
on the $X$-variables, we can deduce that 
in the space $\mathcal{X}$ 
the set of optimal strategies of player~1 is
the projection of a spectrahedron. Similarly, this is true for 
player~2.

Indeed, the set of optimal strategies of each player is not only the projection
of a spectrahedron, but also a spectrahedron. Namely, taking the optimal value of
$v$ (i.e., the value $V$ of the game)
corresponds geometrically to passing over to the intersection with
a separating hyperplane and, because the other additional variables 
(``$T$'')
just refer to a slack matrix,
we see that the set of optimal strategies for
the first player is a spectrahedron. 
In particular, the equation
$V I_n + T = (\langle X, A_{\cdot\cdot kl} \rangle)_{1 \le k,l \le n}$
gives
\[
-V I_n + (\langle X, A_{\cdot\cdot kl} \rangle)_{1 \le k,l \le n} = 
T \succeq 0.
\]
Similar arguments apply for the second player.
Precisely, the spectrahedron for the first player lives in the space
$\sym_n$, whose variables we denote by the symmetric matrix
variable $X$. The inequalities and equations for $X$ are
\[
  \begin{array}{rcl}
    \tr(X) - 1 & = & 0, \\
    -V I_n + (\langle X, A_{\cdot\cdot kl} \rangle)_{1 \le k,l \le n} 
    & \succeq & 0,
  \end{array}
\]
where we can write the equation as
two inequalities and where we can combine all the scalar inequalities
and matrix inequalities into one block matrix inequality.
\end{proof}

\begin{cor}
Explicit LMI descriptions of the sets $\optstrat_1$ and
$\optstrat_2$ of optimal strategies of the two players are
\begin{eqnarray*}
  \optstrat_1 & := & \{X \in \sym_n \, : \, 
   \diag(X, \, 
   -V I_n + (\langle X, A_{\cdot\cdot kl} \rangle)_{1 \le k,l \le n} , \, \\
   & & \; \, \tr(X) - 1, \, 
    1 - \tr(X)) \succeq 0 \}, \\
    \optstrat_2 & := & \{Y \in \sym_n \, : \,
        \diag(Y, \, 
         V I_n + (\langle -A_{ij\cdot\cdot}, Y \rangle)_{1 \le i,j \le n}, \, \\
         & & \; \, \tr(Y)-1, \, 
         1 - \tr(Y)) \succeq 0 \} \, ,
\end{eqnarray*}
where $V$ is the value of the game.
\end{cor}
Note that $\mathcal{O}_1$ and $\mathcal{O}_2$ are spectrahedra in
the space of symmetric matrices.

\begin{example}
\label{ex:zero-sum1}
We consider a semidefinite generalization of a $2 \times 2$-zero-sum matrix game 
known as ``Plus one'' (see, for example, \cite{karlin-peres-2017}). The 
payoff matrix of that bimatrix game is 
$A = \ \left( \begin{array} {cc}
    0 & -1 \\
    1 & 0
    \end{array} \right)$
and for each player, the second strategy is dominant.
Let the semidefinite zero-sum game be defined by
\[
  A_{ijkl} \ = \ \begin{cases}
    1 & \max\{i,j\} > \max\{k,l\}, \\
    -1 & \max\{i,j\} < \max\{k,l\}, \\
    0 & \text{else}
  \end{cases}
\]
for $i,j,k,l \in \{1, 2\}$.
If both players play only diagonal strategies (i.e., $X$ and $Y$ are diagonal matrices),
then the payoffs correspond to the payoffs of the underlying matrix game.
By Lemma~\ref{le:antisymm}, the value of the semidefinite game is 0. 
To determine an optimal
strategy for player~1, we consider those points in the feasible set of the
SDP~\eqref{eq:sdp-minmax1} in Lemma~\ref{le:apriori1}, which have
$v_1 = 0$. Hence, we are looking for matrices $T, X \succeq 0$ such that
\[
  T \ = \ \left( \begin{array}{cc}
    \langle X,A_{\cdot\cdot 11} \rangle & \langle X,A_{\cdot \cdot 12} \rangle \\
    \langle X,A_{\cdot\cdot 21} \rangle & \langle X,A_{\cdot \cdot 22} \rangle
  \end{array} \right) 
  \ = \ \left( \begin{array}{cc}
    x_{12} + x_{21} + x_{22} & - x_{11} \\
    - x_{11} & - x_{11}
  \end{array} \right).
\]
Since $X \succeq 0$ implies $x_{11} \ge 0$ and $T \succeq 0$ implies $-x_{11} \ge 0$,
we obtain $x_{11} = 0$. Hence, $x_{22} = 1 - x_{11} = 1$. Further, $X \succeq 0$
yields $x_{12} = x_{21} = 0$. Therefore, the optimal strategy of player~1 is
$\left( \begin{array}{cc}
  0 & 0 \\
  0 & 1
  \end{array} \right)$, and, similarly, the optimal strategy of player~2 is the same one.
\end{example}

\begin{example}
Consider a slightly different version of Example~\ref{ex:zero-sum1}, in which
the optimal strategies are not diagonal strategies.
For $i,j,k,l \in \{1, 2\}$, let
\[
  A_{ijkl} \ = \ (i+j) - (k+l).
\]
By Lemma~\ref{le:antisymm}, the value of the game is 0. If both players
play only diagonal strategies, then the payoffs coincide (up to a factor
of~2 in the payoffs, which is irrelevant for the optimal strategies) with
the payoffs of the underlying zero-sum matrix game.
Interestingly, we show that  the optimal strategies are not diagonal strategies 
here. As in 
Example~\ref{ex:zero-sum1}, we determine an optimal strategy for player~1
by considering the feasible points of the SDP~\eqref{eq:sdp-minmax1} with
$v_1 = 0$. Hence, we search for $T, X \succeq 0$ such that
\[
  T \ = \ \left( \begin{array}{cc}
    x_{12} + x_{21} + 2 x_{22} & - x_{11} + x_{22} \\
    - x_{11} + x_{22} & - 2 x_{11} - x_{12} - x_{21}.
  \end{array} \right).
\]
Since, using the symmetry $x_{12} = x_{21}$, we have
$
  \det T \ = \ - (x_{11} + 2 x_{12} + x_{22})^2,
$
we see that $T \succeq 0$ implies $x_{12} = - \frac{1}{2}(x_{11} + x_{22})$.
Hence,
\[
  X = \left( \begin{array}{cc}
  x_{11} & - \frac{1}{2}(x_{11} + x_{22}) \\
  - \frac{1}{2}(x_{11} + x_{22}) & x_{22}
  \end{array}\right),
\]
so that $X \succeq 0$ implies in connection with the arithmetic-geometric inequality
that $x_{11} = x_{22} = \frac{1}{2}$. Thus, the optimal strategy of player~1 is
$ \frac{1}{2}
\left( \begin{array}{cc}
  1 & - 1 \\
  - 1 & 1
  \end{array} \right)$, and, similarly, the optimal strategy of player~2 is the same one.
  
Note that $X = \left( \begin{array}{cc}
  0 & 0 \\
  0 & 1
  \end{array} \right)$ is not an optimal strategy for player~1, since, for example,
for a given strategy $Y$ of player~2, the payoff is
$
  2 y_{11} + 2 y_{12}.
$
Specifically, the choice 
$Y = \frac{1}{4} \left( \begin{array}{cc}
  1 & - \sqrt{3} \\
  - \sqrt{3} & 3
\end{array} \right)
$ 
yields a payoff of 
$\frac{1}{2}(1-\sqrt{3}) < 0$ for player~1.
\end{example}

\section{The almost-equivalence of semidefinite zero-sum games and 
semidefinite programs\label{se:equivalence}}

We give a semidefinite generalization of Dantzig's almost equivalence of
zero-sum matrix games and linear programming
\cite{dantzig-1951-equivalence}, see also \cite{adler-2013}.
Given an LP in the form
\[
\min_x \{  c^Tx \, : \ Ax \ge 
  b, \, x \ge 0 \}
\]
with $A\in \R^{m\times n}$ and $ b\in \R^m$ and $c \in \R^n$,
Dantzig constructed a zero-sum matrix game with the payoff matrix
\[
  \left( \begin{array}{ccc}
    0         & A & -b \\
    -A^T & 0 & c \\
    b^T & -c^T & 0
  \end{array} \right).
\]
For the semidefinite generalization,
the following variant of a duality statement is convenient, whose
proof is given for the sake of completeness.

\begin{lemma}
\label{le:SDP-modif-nf}
For $A_1, \ldots, A_m, C \in \sym_n$ and $b \in \R^m$,
 the following SDPs in the slightly modified normal forms 
constitute a primal-dual pair.
\begin{eqnarray}
  \label{eq:sdp-form3}
  & & \inf_X \{ \langle C, X \rangle \, : \ \langle A_i, X \rangle \ge 
  b_i, \, 1 \le i \le m, \, X \succeq 0 \} \\
  \label{eq:sdp-form4}
  & \text{ and } &
  \sup_{y,S} \{ b^T y\, : \, \sum_{i=1}^m y_i A_i + S = C, \, y \in \R_+^m, \,
S \succeq 0\}.
\end{eqnarray}
We call them \emph{SDPs in modified primal and dual forms}.
\end{lemma}

\begin{proof}
We derive these forms from the usual forms, in which the primal
contains a relation ``$=$'' rather than a relation ``$\ge$'' and 
in which the dual uses
an unconstrained variable vector $y$ rather than a non-negative vector.
Starting from the standard pair, we extend each matrix $A_i$ to
a symmetric $2n \times 2n$-matrix via a single additional
non-zero element, namely $-1$, in entry $(n+i,n+i)$ of the modified
$A_i$, $1 \le i \le m$.
Moreover, we formally
embed $C$ into a symmetric $2n \times 2n$-matrix. The dual
(in the original sense) of the extended problem still uses a vector
of length $y$, but the modifications in the primal problem 
give the additional conditions $y_i \ge 0$, $1 \le i \le m$. This shows the
desired modified forms of the primal-dual pair.
\end{proof}

A semidefinite zero-sum game with 
$\mathcal{X}=\mathcal{Y}$
is called \emph{symmetric} if the payoff tensor
$A$ satisfies the skew symmetric relation
$
  A_{ijkl} = - A_{klij}.
$
By Lemma~\ref{le:antisymm},
the value of a symmetric game with payoff tensor $A$ on the strategy
space $\sym_n \times \sym_n$ is zero.
Therefore, there
exists a strategy $\bar{X}$ of the first player such that
\begin{equation}
  \label{eq:opt-strat1}
  \langle ( \langle \bar{X}, A_{\cdot \cdot kl} \rangle)_{1 \le k,l \le n}, Y \rangle \ge 0
  \quad \text{for all } Y \in \sym_n \text{ with }
  \tr(Y) = 1, \ Y \succeq 0.
\end{equation}
Generalizing the notion in \cite{adler-2013}, we call such a strategy 
$\bar{X}$ a \emph{solution} of the 
symmetric game. The condition~\eqref{eq:opt-strat1} states
that the matrix 
$\langle ( \langle \bar{X}, A_{\cdot \cdot kl} \rangle)_{1 \le k,l \le n}$
is contained in the dual cone of $\mathcal{S}_n^+$.
Since the cone $\mathcal{S}_n^+$ of positive semidefinite matrices
is self-dual, \eqref{eq:opt-strat1} translates to
\begin{equation}
  \label{eq:solution-notion}
  (\langle \bar{X}, A_{\cdot \cdot kl} \rangle)_{1 \le k,l \le n}
  \succeq 0.
\end{equation}

It is useful to record the following specific version of a symmetric
minimax theorem, which is a special case of the minimax theorem for
convex games \cite{dresher-karlin-1953} and of the minimax theorem
for quantum games \cite{jain2009parallel}.

\begin{lemma}[Minimax theorem for symmetric semidefinite zero-sum
games]
Let $G$ be a symmetric semidefinite zero-sum game with
payoff tensor $A$. Then there exists a solution
strategy $\bar{X}$, i.e., a matrix $\bar{X} \in \sym_n$
satisfying~\eqref{eq:solution-notion}.
\end{lemma}

\begin{proof}
There exists a strategy $\bar{X}$ satisfying~\eqref{eq:opt-strat1}.
By the considerations before the theorem, 
we obtain~\eqref{eq:solution-notion}.
\end{proof}

Now we generalize the Dantzig construction. 
Given an SDP in modified normal form of Lemma~\ref{le:SDP-modif-nf},
we define the following \emph{semidefinite Dantzig game}
on $\mathcal{S}_{n+m+1} \times \mathcal{S}_{n+m+1}$.
The strategies of both players can be viewed
as positive semidefinite block matrices 
$\diag(X,y,t)$ with $X \in \sym_n^+$, $y \in \R_+^m$, 
$t \in \R_+$ and trace~1.  

The payoff tensor $Q$ is defined as follows.
For $1 \le k,l \le n$ and $1 \le j \le m$, let
\[
  Q_{k,l,n+j,n+j} \ = \ - Q_{n+j,n+j,k,l} \ = \ (A_j)_{kl} \, .
\]
For $1 \le k,l \le n$, let
\[
    Q _{n+m+1,n+m+1,k,l} 
   \ = \ - Q_{k,l,n+m+1,n+m+1}
   \ = \ C_{kl} \, .
\]
For $1 \le i \le m$, let
\[
  Q_{n+i,n+i,n+m+1,n+m+1} 
  \ = \ 
  -Q _{n+m+1,n+m+1,n+i,n+i} 
  \ = \ b_i \, .
\]
All other entries in the payoff tensor $Q$ are zero. Note that $Q$ has a block
structure: For every non-zero entry $Q_{ijkl}$, we have either
$i,j \in \{1, \ldots, n\}$ or $i,j \in \{n+1, \ldots, n+m\}$
or $i=j=n+m+1$. An analogous property holds for $k,l$.

Let $\diag(\bar{X}, \bar{y}^T, \bar{t})^T$ be a solution to
this symmetric game. By the definition of a solution, we have
\begin{equation}
  \label{eq:dantzig-solution}
  ( \langle \diag(\bar{X}, \bar{y}, \bar{t}),
  Q_{\cdot \cdot kl} \rangle)_{1 \le k,l \le n + m+1} \ \succeq \ 0.
\end{equation}
Due to the block structure of the payoff tensor $Q$, the matrix on the left-hand side 
of~\eqref{eq:dantzig-solution} has a block structure as well.
The upper left $n \times n$-matrix in~\eqref{eq:dantzig-solution}
gives the condition $- \sum_{j=1}^n \bar{y}_j A_j + \bar{t}C \succeq 0$.
The square submatrix in~\eqref{eq:dantzig-solution} indexed by
$n+1, \ldots, n+m$ is a diagonal matrix and gives
the conditions
$\langle A_i, \bar{X} \rangle - b_i \bar{t} \ge 0$, $1 \le i \le m$.
The right lower entry in~\eqref{eq:dantzig-solution} gives
$b^T \bar{y} - \langle C, \bar{X} \rangle \ge 0$.
Hence, \eqref{eq:dantzig-solution} is equivalent to the system
\begin{eqnarray}
\langle A_i, \bar{X} \rangle - b_i \bar{t} & \ge & 0, \quad 1 \le i \le m, 
  \label{eq:dantzig5} \\
- \sum_{j=1}^m \bar{y}_j A_j + \bar{t} C & \succeq & 0, 
  \label{eq:dantzig6} \\
b^T \bar{y} - \langle C, \bar{X} \rangle & \ge & 0,
  \label{eq:dantzig7}
\end{eqnarray}
and in addition, we have the conditions defining a strategy,
\[
 \bar{y} \ge 0, \; \bar{X} \succeq 0, \; \bar{t} \ge 0 \text{ and } 
  \mathbf{1}^T \bar{y} + \tr(\bar{X}) + \bar{t} \ = \ 1.
\]

This allows us to state the following result on the almost equivalence
of semidefinite zero-sum games and semidefinite programs.
Recall that reducing the equilibrium
problem in a semidefinite zero-sum game to
semidefinite programming follows from Section~\ref{se:zero-sum}.

\begin{thm}
 \label{th:equivalence1}
The following holds for the semidefinite Dantzig game:
\begin{enumerate}
\item $\bar{t} (b^T \bar{y} - \langle C, \bar{X} \rangle) = 0$.
\item If $\bar{t} > 0$, then $\bar{X} \bar{t}^{-1}$, 
$\bar{y} \bar{t}^{-1}$ and some corresponding slack matrix
$\bar{S}$
are an optimal solution to the primal-dual SDP pair given 
in \eqref{eq:sdp-form3} and \eqref{eq:sdp-form4}.
\item If $b^T \bar{y} - \langle C, \bar{X} \rangle > 0$, then the primal problem
  or the dual problem is infeasible.
\end{enumerate} 
\end{thm}

The theorem ignores the case $\bar{t} = 0$. In the special case of bimatrix games, that exception was already
observed in Dantzig's treatment \cite{dantzig-1951-equivalence}
and overcome by Adler \cite{adler-2013} and von 
Stengel \cite{von-stengel-2022}.

While the precondition in (3) looks like a statement of 
missing strong duality, note that $(\bar{X}, \bar{y})$ does
not satisfy the constraints in the initially stated primal-dual SDP pair.
If the initial SDP pair does not have an optimal primal-dual
pair, then clearly case
 (2) can never hold. This is a 
difference to the LP case, where, say, in the case of finite optimal
values, there always exists an optimal primal-dual pair and so
case (2) is never ruled out a priori in the same way.
This qualitative difference was expected, because for a semidefinite
zero-sum game, the corresponding primal and dual feasible regions have relative
interior points
and thus there exists
an optimal primal-dual pair. So, the qualitative situation reflects
that for semidefinite zero-sum games, 
the SDPs characterizing the optimal strategies are always well behaved.

By adding the precondition that the original pair of SDPs has
primal-dual interior points, we come into the same situation that
case (2) is not ruled out a priori. 

\begin{proof}
Since 
$\bar{X}$ and $\bar{y}$ are feasible solutions of the SDPs 
in~\eqref{eq:dantzig5} and~\eqref{eq:dantzig6}, the weak 
duality theorem for semidefinite programming implies
\[
  \bar{t} ( b^T \bar{y} - \langle C, \bar{X} \rangle) \ \le \ 0.
\]
Since $\bar{t} \ge 0$, we obtain
$\bar{t} = 0$ or $b^T \bar{y} -\langle C, \bar{X} \rangle \le 0$.
In the latter case,~\eqref{eq:dantzig7} implies 
$b^T \bar{y} - \langle C, \bar{X} \rangle = 0$.
Altogether, this gives $\bar{t} (b^T \bar{y} - \langle C, \bar{X} \rangle) = 0$.

For the second statement, let $\bar{t} > 0$. Then 
$\bar{X} \bar{t}^{-1}$, $\bar{y} \bar{t}^{-1}$ and the 
corresponding slack matrix $\bar{S}$ give a feasible point
of the primal-dual SDP pair stated initially. Since
$b^T \bar{y} - \langle C, \bar{X} \rangle = 0$ and thus
$b^T (\bar{y}\bar{t}^{-1}) - \langle C, \bar{X} \bar{t}^{-1} \rangle = 0$,
this feasible point is an optimal solution.

For the third statement, let 
$b^T \bar{y} - \langle C, \bar{X} \rangle > 0$.
Then statement (1) implies $\bar{t} = 0$. Thus,
$\langle A_i, \bar{X} \rangle \ge 0$ for $1 \le i \le m$ and
$\sum_{j=1}^m \bar{y_j} A_j \preceq 0$.
Since $b^T \bar{y} - \langle C, \bar{X} \rangle > 0$,
we obtain 
$b^T \bar{y} >0$ or 
$\langle C, \bar{X} \rangle < 0$.

In the case $\langle C, \bar{X} \rangle < 0$,
assume that the originally stated primal~\eqref{eq:sdp-form3} has a feasible 
solution $X^{\Diamond}$.
Then, for any $\lambda \ge 0$, the point
$X^{\Diamond} + \lambda \bar{X}$
is a feasible solution as well.
By considering $\lambda \to \infty$, 
we see that~\eqref{eq:sdp-form3}
has optimal value $-\infty$. By the weak duality
theorem, the dual problem~\eqref{eq:sdp-form4} 
cannot be feasible.
In the case $b^T \bar{y} >0$, similar arguments show that the primal
problem is infeasible.
\end{proof}

Note that in Theorem~\ref{th:equivalence1} it is not necessary
 to assume that the constraints of
the SDP are linearly independent, since we have not expressed
the situation only in terms of the slack variable.

\begin{example}
Consider the SDP given in the primal normal form
$$\min_X \left\{ \left\langle \begin{pmatrix}
2 & 0\\
0 & 2
\end{pmatrix}, X \right\rangle :     \left\langle \begin{pmatrix}
1 & 0\\
0 & 1
\end{pmatrix}, X \right\rangle \geq 1, X\succeq 0 \right\} $$
and its dual
$$
\max_{y,S} \left\{ y:~ y \begin{pmatrix}
1 & 0\\
0 & 1
\end{pmatrix} +S = \begin{pmatrix}
2 & 0\\
0 & 2
\end{pmatrix}, \, y\geq 0, \, S\succeq 0 \right\}.
$$
One can easily verify that optimal solutions of the SDP in primal normal form are matrices $X'\in \psd_2$ with $\tr(X')=1$ and the only optimal solution of the dual problem is the pair $(y',S')$, where $y'=2$ and $S'=0\in \psd_2.$
The payoff tensor $Q$ in the corresponding Dantzig game in flattened form is
$$
\left( \begin{array}{rrr|r|r} 
0 & 0 &0 & 1 &-2\\
0 &0 &0 & 1 &-2\\
0&0&0&0&0\\ \hline
-1 & -1 & 0 &0& 1\\ \hline
2&2&0&-1&0
\end{array} \right),
$$
where the rows and columns are indexed by 
$X_{11}$,$X_{22}$,$2X_{12}$,$y_1$,$t$.
To extract an optimal strategy $\diag(\bar{X},\bar{y},\bar{t})$ for this game, we observe that \eqref{eq:dantzig5} gives $\bar{x}_{11} + \bar{x}_{22} \geq \bar{t}$ and \eqref{eq:dantzig6} yields $2\bar{t}  \geq \bar{y}$.
Then \eqref{eq:dantzig7} implies $\bar{y} \geq 2(\bar{x}_{11} + 
\bar{x}_{22})$ and we obtain 
$\bar{x}_{11} + \bar{x}_{22} = \bar{t} = \frac{\bar{y}}{2}$.
By the trace condition, an optimal strategy is of the form 
$\diag(\bar{X}, \frac{1}{2},\frac{1}{4})$, 
where $\bar{X}\in \psd_2$ satisfies $\tr(\bar{X})=\frac{1}{4}$.

Since $\bar{t} > 0$, Theorem \ref{th:equivalence1} implies that
$4\bar{X}$ and $ 4\bar{y}$ are optimal solutions to the primal-dual 
SDP pair. 
\end{example}

\section{General semidefinite 
games\label{se:general-semidefinite-games}}

Now we study two-player semidefinite games without the zero-sum condition.
By Glicksberg's result \cite{glicksberg-1952} (see also Debreu \cite{debreu-1952}
and Fan \cite{fan-1952}), there always exists a Nash equilibrium
for these games. 
This is so because they are a special case of $N$-players
continuous games with continuous payoff functions
defined on convex compact Hausdorff spaces \cite{glicksberg-1952}. 
The goal of this section is to provide a characterization of
the Nash equilibria in terms of spectrahedra, see 
Theorem~\ref{th:nash-charact1}.

Recall the following representation of Nash equilibria for bimatrix games
in terms of polyhedra, as introduced by 
Mangasarian \cite{mangasarian-1964} (see also \cite{von-stengel-2002}):

\begin{definition} \label{de:mangasarian}
For an $m \times n$-bimatrix game $(A,B)$, let the
polyhedra $P$ and $Q$ be defined by
\begin{eqnarray}
  \: \quad P & = & \{(x, v) \in \R^m \times \R \, : 
  x \ge 0
  , \; 
    x^T B \le {\bf 1}^T v
  , \; {\bf 1}^T x = 1 \} \, \label{eq:p} \, , \label{eq:p1} \\
  \: \quad Q & = & \{(y, u) \in \R^n \times \R \, :  
    A y \le {\bf 1} u
 , \;
 y \ge 0
 ,
 \; {\bf 1}^T y = 1 \} \label{eq:q} \, , \label{eq:q1}
\end{eqnarray}
where $\mathbf{1}$ denotes the all-ones vector.
\end{definition}
In $P$, the inequalities $x \ge 0$ are numbered by
$1, \ldots, m$ and the inequalities $x^T B \le {\bf 1}^T v$ are
numbered by $m+1, \ldots, m+n$.
In $Q$, the inequalities $A y \le {\bf 1} u$ are numbered
by $1, \ldots, m$ 
and the inequalities $y \ge 0$ are numbered by
$m+1, \ldots, m+n$.
In this setting, a pair of mixed strategies
$(x,y) \in \Delta_1 \times \Delta_2$
is a Nash equilibrium if and only if there exist
$u, v \in \R$ such that
$(x,v) \in P$,
$(y,u) \in Q$ and for all $i \in \{1, \ldots, m+n\}$,
the $i$-th inequality of $P$ or $Q$ is binding (i.e., it holds with equality).
Here, $u$ and $v$ represent the payoffs of player~1 and player~2, respectively.
This representation allows us to study Nash equilibria in terms of pairs of points in 
$P \times Q$.

We aim at a suitable generalization of this combinatorial
characterization to the case of semidefinite games. Note that in the case of
bimatrix games, the characterization is strongly based on the finiteness of
the pure strategies; this does not hold anymore for semidefinite games.
Therefore, we start from equivalent versions of the bimatrix game 
polyhedra~\eqref{eq:p1} and~\eqref{eq:q1}, which do not use
finitely many pure strategies in their formulation. Instead, the best responses
are expressed more explicitly as a maximization problem,
\begin{eqnarray*}
  P & = & \{(x, v) \in \R^m \times \R \, : x \ge 0, \;
  \max _y \{ x^T B y \, : \, {y} \in \Delta_2 \} \le v, \;
  {\bf 1}^T x = 1 \} \, \label{eq:p2} , \\
   Q & = & \{(y, u) \in \R^n \times \R \, : y \ge 0, \; 
   \max _x \{x^T A y \, : \, x \in \Delta_1\} \le u, \;
   {\bf 1}^T y = 1 \} \label{eq:q2} .
\end{eqnarray*}

While the generalizations to semidefinite games are no longer polyhedral,
it is convenient to keep the symbols $P$ and $Q$ for the notation. 
Consider the sets
\begin{eqnarray*}
  P & = & \{(X, v) \in \sym_m \times \R \, : X \succeq 0, \;
  \max _Y \{ p_B(X,Y) \, : \, Y \in \sym_n^+, \, \tr(Y) = 1\} \le v, \;
    \label{eq:p3} \\
    & & \; \, \tr(X) = 1 \} \, , \\
  Q & = & \{(Y, u) \in \sym_n \times \R \, : Y \succeq 0, \;
   \max _X \{p_A(X,Y) \, : \, X \in \sym_m^+, \, \tr(X) = 1\} \le u,  \;
   \label{eq:q3} \\
   & & \; \,
   \tr(Y) = 1 \} .
\end{eqnarray*}

We show that $P$ and $Q$ are spectrahedra in the spaces
$\sym_m \times \R$ and $\sym_n \times \R$.
 Similar to the considerations in the zero-sum case, for a fixed
 $X$, the expression $\max \{ p_B(X,Y) \, : \, Y \in \sym_n^+, \, \tr(Y) = 1\}$
 can be rewritten as
 \begin{eqnarray*}
 &  & - \min \{- p_B(X,Y) \ : \ Y \in \sym_n^+, \, \tr(Y) = 1 \} \\
 & = & - \min \{ \langle ( \langle X, - B_{\cdot \cdot k,l} \rangle )_{1 \le k,l \le n}, Y \rangle \ : \
   \tr(Y) = 1, \, Y \succeq 0\} \\
 & = & - \max\{ 1 \cdot v_1 \ : \ v_1 I_n + T 
        = (\langle X, -B_{\cdot \cdot k,l} \rangle )_{1 \le k,l \le n}, \
          T \succeq 0, \ v_1 \in \R \} \\
& = & \min \{ - v_1 \ : \ v_1 I_n + T 
        = (\langle X, -B_{\cdot \cdot k,l} \rangle )_{1 \le k,l \le n}, \
          T \succeq 0, \ v_1 \in \R \} \\
& = & \min \{ v_1 \ : \ - v_1 I_n + T 
        = (\langle X, -B_{\cdot \cdot k,l} \rangle )_{1 \le k,l \le n}, \
          T \succeq 0, \ v_1 \in \R \}.
\end{eqnarray*} 
Here, the minimum is attained, since the feasible set in the first line of the equations
is compact. The maximum in the third line is attained due to the strong duality theorem
for semidefinite programming and using that the feasible set in the first two lines
($\{Y \ : \ Y \succeq 0, \ \tr(Y) = 1$\}) has a strictly interior point and thus satisfies Slater's condition.
Hence,
 \begin{eqnarray*}
  P & = & \{(X, v) \in \sym_m \times \R \, : \\
  & & X \succeq 0, \; 
     \min_{T, \ v_1} \{ v_1 \, : \, - v_1 I_n + T = (\langle X,-B_{\cdot\cdot kl}  \rangle)_{1 \le k,l \le n}, \, T \succeq 0, v_1 \in \R \} 
     \ \le \ v, \\
  & & \tr(X) = 1 \} \, \label{eq:p4} \, , \\
   Q & = & \{(Y, u) \in \sym_n \times \R \, : \\
   & & Y \succeq 0, \;
   \min_{S, u_1} \{ u_1 \, : \, - u_1 I_n + S = (\langle -A_{ij\cdot\cdot},
    Y \rangle)_{1 \le i,j \le n}, \, S \succeq 0, u_1 \in \R \} 
   \ \le \ u, \,
   \\
   & & \tr(Y) = 1 \} \label{eq:q4} \, .
\end{eqnarray*}
If the $\min$-problem inside $P$ has some feasible solution 
$(v_1,T)$ then for any $v_1' \ge v_1$, there exists a feasible solution
$(v_1',T')$ as well. Namely, set $T':=T+(v_1'-v_1) I_n \succeq 0$.
Thus we have
\begin{eqnarray*}
  P & = & \{(X, v) \, : \,
  X \succeq 0, \, 
     -v I_n + T = (\langle X,-B_{\cdot\cdot kl}  \rangle)_{1 \le k,l \le n}, \, T \succeq 0, \, \tr(X) = 1 \} \, \label{eq:p5} \, , \\
   Q & = & \{(Y, u) \, : \,
   Y \succeq 0, \,
   -u I_n + S = (\langle -A_{ij\cdot\cdot}, Y \rangle)_{1 \le i,j \le n}, \, S \succeq 0, \,
   \tr(Y) = 1 \} \label{eq:q5} \, .
\end{eqnarray*}
We claim that $P$ and $Q$ are spectrahedra in the spaces $\sym_m \times \R$ and
$\sym_n \times \R$.
 For $P$, the inequalities and equations are given by
\begin{equation*}
   (\langle X,-B_{\cdot\cdot kl}  \rangle)_{1 \le k,l \le n} + v I_n \succeq 0, 
   \; \, X \succeq 0, \; \, \tr(X)  = 1,
\end{equation*}
where the equation can be written as
two inequalities and where we can combine all the scalar inequalities
and matrix inequalities into one block matrix inequality.
The spectrahedra $P$ and $Q$
can be used to provide the following characterization of Nash equilibria
in terms of a pair of projections of spectrahedra. We build on the
terminology from the bimatrix situation after 
Definition~\ref{de:mangasarian} and describe Nash equilibria
together with their payoffs.

\begin{thm}
\label{th:nash-charact1} 
A quadruple $(X,Y,u,v)$ represents a Nash equilibrium of the
semidefinite game if and only if 
\begin{enumerate}
\item $(X,v) \in P$,
\item $(Y,u) \in Q$,
\item and in every finite rank-1 decomposition
of $X$ and $Y$,
\[
  X \ = \ \sum_s \lambda_s p^{(s)} (p^{(s)})^T, \quad Y \ = \ \sum_t \mu_t q^{(t)} (q^{(t)})^T
\]
with $\lambda_s, \mu_t > 0$, $\tr(p^{(s)} (p^{(s)})^T) = 1$, $\tr(q^{(t)} (q^{(t)})^T) = 1$
and $\sum_s \lambda_s = \sum_t \mu_t = 1$, we have
\begin{eqnarray}
  \label{eq:charact1}
  \langle ( \langle X, B_{\cdot \cdot kl} \rangle)_{1 \le k,l \le n}, q^{(t)} (q^{(t)})^T \rangle 
  & = & v \; \text{ for all $t$} \\
  \label{eq:charact2}
\text{and} \qquad
  \langle p^{(s)} (p^{(s)})^T, ( \langle A_{ij \cdot \cdot}, Y \rangle)_{1 \le i,j \le m} \rangle & = & u
  \; \text{ for all $s$}.
\end{eqnarray}
\end{enumerate}
\end{thm}

The positivity condition on $\lambda_i, \mu_j$ reflects the binding property of 
$x_i \ge 0$ or $y_j \ge 0$ from the bimatrix situation.
Moreover, in the bimatrix situation the inequalities are induced by the pure strategies,
which are the extreme points of $\Delta_1$ and $\Delta_2$.
In the semidefinite situation, we can associate with each extreme point of the strategy space
an inequality, namely the inequality (say, for $P$ and an extreme point $Y$ of the strategy space
of the second player)
\[
  p_B(X,Y)  \ \le \ v.
\]

\begin{proof}
Let $(X,Y,u,v)$ represent a Nash equilibrium. Then $X$ is a best response of $Y$ and $Y$ is
a best response of $X$, so that by definition of $P$ and $Q$, we have $(X,v) \in P$ and
$(Y,u) \in Q$.
Let $\sum_s \lambda_s p^{(s)} (p^{(s})^T$ be a finite rank 1-decomposition of $X$
with $\lambda_s > 0$, $\tr(p^{(s)} (p^{(s)})^T) = 1$ and
$\sum_s \lambda_s = 1$. Then
\begin{eqnarray*}
  p_A(X,Y) & = & \langle X, \langle (A_{ij \cdot \cdot }, Y \rangle)_{1 \le i,j \le m} \rangle \\
    & = & \sum_{s} \lambda_s \langle p^{(s)} (p^{(s)})^T, (\langle A_{ij \cdot \cdot }, Y \rangle)_{1 \le i,j \le m} \rangle.
\end{eqnarray*}
Since the first player's payoff is $u$, the best response property gives 
$p_A(p^{(s)} (p^{(s)})^T, Y) \le u$. If one of 
$p^{(s)} (p^{(s)})^T$ had $p_A(p^{(s)} (p^{(s)})^T, Y) < u$, then, 
since $X$ is a convex combination,
we would have $p_A(X,Y) < u$, a contradiction.
The statement on $Y$ follows similarly.

Conversely, let $(X,v) \in P$,
$(Y,u) \in Q$ and in every finite rank-1 decomposition
$
  X \ = \ \sum_s \lambda_s p^{(s)} (p^{(s)})^T$, 
$Y \ = \ \sum_t \mu_t q^{(t)} (q^{(t)})^T
$
with $\lambda_s, \mu_t > 0$, $\tr(p^{(s)} (p^{(s)})^T) = 1$, $\tr(q^{(t)} (q^{(t)})^T) = 1$
and $\sum_s \lambda_s = \sum_t \mu_t = 1$, we have~\eqref{eq:charact1} and~\eqref{eq:charact2}.
Since $(Y,u) \in Q$, we have $u^*:=\max \{ p_A(X,Y) \ : \ X \in \sym_m^+,  \ \tr(X) = 1\} \le u$. 
Due to~\eqref{eq:charact2}, this gives
$p_A(p^{(s)} (p^{(s)})^T,Y) = u$ for all $s$ and thus
$p_A(p^{(s)} (p^{(s)})^T,Y) = u^* = u$ for all $s$.
Hence, $p_A(X,Y) = u^* = u$ and $X$ is a best response to $Y$.
Similarly, $Y$ is a best response to $X$ with payoff $v$ so that altogether 
$(X,Y,u,v)$ represents a Nash equilibrium.
\end{proof}

\begin{rem}
Theorem~\ref{th:nash-charact1} also holds if 
``in every finite rank-1 decomposition'' is replaced by
``in at least one finite rank-1 decomposition". The proof also works in that
setting. We will illustrate this further below, in Example~\ref{ex:5nasheqb}.
\end{rem}

Compared to bimatrix games, the more general situation of semidefinite
games is qualitatively different in the following sense. In a bimatrix game,
every mixed strategy is a unique convex combination of the pure strategies.
In a semidefinite game, the analogs of pure strategies are the 
rank 1-matrices in the spectraplex and the decompositions of the mixed
(i.e., of rank at least 2)
strategies as convex combinations of rank-1 matrices are no longer unique.
However, the situation for a Nash equilibrium to have several decompositions is quite restrictive.

\begin{rem}
\label{re:decomposition-rank1}
To obtain a decomposition of a positive semidefinite matrix with 
trace~1 into the sums of
positive semidefinite rank 1-matrices with trace~1, one can proceed
as follows. Consider a  spectral decomposition. Then replace the 
eigenvalues (which sum to 1) by eigenvalues 1 and interpret the
 original eigenvalues as coefficients of a convex combination.
 \end{rem}

\begin{example}
We consider the example of a hybrid game, where the first player plays 
a strategy in the simplex
$\Delta_2$ and the second player plays a strategy in $\sym_2^+$
with trace~1; note that the hybrid game can be encoded into a semidefinite
game on $\sym_2 \times \sym_2$ by setting $A_{ijkl} = B_{ijkl} = 0$
whenever $i \neq j$.
We can describe the situation in terms of an index $i$ for the
first player and the indices $(j,k)$ for the second player.
For $i=1$, let
\[
  (A_{1jk})_{1 \le j,k \le 2} \ = \ (B_{1jk})_{1 \le j,k \le 2} = \left( \begin{array}{cc}
    1 & \varepsilon \\
    \varepsilon & 0
    \end{array} \right),
\]
and for $i=2$, let
\[
  (A_{2jk})_{1 \le j,k \le 2} \ = \ (B_{2jk})_{1 \le j,k \le 2} = \left( \begin{array}{cc}
    0 & \varepsilon \\
    \varepsilon & 1
    \end{array} \right).
\]
To determine the Nash equilibria, we consider three cases:

\emph{Case 1: The first player plays the first pure strategy.}
Then the second player has payoff 
$y_{11} + 2 \varepsilon y_{12}$.
A small computation shows that for small $\varepsilon$,
the second player's best response is
\[
  \frac{1}{2\sqrt{4 \varepsilon^2+1}}
  \left( \begin{array}{cc}
    \sqrt{4 \varepsilon^2+1}+1 & 2 \varepsilon \\
    2 \varepsilon & \sqrt{4 \varepsilon^2+1}-1
  \end{array} \right),
\]
and indeed, this gives a Nash equilibrium.

\emph{Case 2: The first player plays the second pure strategy.}
Analogously, the second player's best response is
\[
  \frac{1}{2\sqrt{4 \varepsilon^2+1}}
  \left( \begin{array}{cc}
    \sqrt{4 \varepsilon^2+1}-1 & 2 \varepsilon \\
    2 \varepsilon & \sqrt{4 \varepsilon^2+1}+1
  \end{array} \right),
\]
and indeed, this gives a Nash equilibrium.

\emph{Case 3: The first player plays a totally mixed strategy.}
Let $x=(x_1,x_2) = (x_1,1-x_1) \in \Delta_2$ with $x_1,x_2 > 0$ and $x_1+x_2=1$.
The second player's best response has the payoff
\begin{eqnarray*}
& & \max \{x_1 y_{11} + 2 \varepsilon x_1 y_{12} + (1-x_1) y_{22} + 2 \varepsilon (1-x_1) y_{12} 
  \ : \ Y \succeq 0, \, \tr(Y) = 1 \} \\
& = & \max \{x_1 y_{11} + (1-x_1) y_{22} + 2 \varepsilon y_{12}
  \ : \ Y \succeq 0, \, \tr(Y) = 1 \}.
\end{eqnarray*}
The first pure strategy would give for the first player
$
  y_{11} + 2 \varepsilon y_{12}
$
and the second pure strategy would give for the first player
$
   y_{22} + 2 \varepsilon y_{12}.
$
If $(x,Y)$ is a Nash equilibrium such that $x$ is a totally mixed
strategy, we must have equality, that is,
$
  y_{11} + 2 \varepsilon y_{12} 
  \ = \ y_{22} + 2 \varepsilon y_{12}.
$
Hence, $y_{11} = y_{22}$ and the payoff of the second player is
\[
 \max \{ x_1 y_{11} + (1-x_1) y_{11} + 2 \varepsilon y_{12} \} \\
 \ = \ \max \{ y_{11} + 2 \varepsilon y_{12} \},
\]
which has become independent of $x_1$.
The payoff of the second player is maximized for the value 
$y_{12} = \sqrt{y_{11} y_{22}} = y_{11}$. Thus, the payoff for the second player is
\[
  \max \{y _{11} + 2 \varepsilon y_{11}\} \\
  \ = \ \max \{ (1+2 \varepsilon) y_{11} \}.
\]
Hence, for every non-negative $\varepsilon$, the best response of the second
player is
\begin{equation}
  \label{eq:onehalfs}
  \left( \begin{array}{cc}
    1/2 & 1/2 \\
    1/2 & 1/2
  \end{array} \right).
\end{equation}
As apparent from the above considerations, in that case both 
pure strategies of the first player
are best responses of the second player.
To determine the strategy $x$ of the first player, we use the
condition that
the maximum
$\max_{Y \in \mathcal{Y}} \{x_1 y_{11} + (1-x_1) y_{22} + 2 \varepsilon
  \sqrt{y_{11} y_{22}}\}$ has to be attained at the matrix \eqref{eq:onehalfs}.
  Substituting $y_{22} = 1-y_{11}$, the resulting univariate problem in 
  $y_{11}$
  gives $x = (1/2,1/2)$.
 The payoff is 
$\frac{1}{2} + \varepsilon$ for both players.
\end{example}

We close the section by mentioning that some classical results
for bimatrix games remain true for semidefinite 
games. Since semidefinite games are convex compact games, 
the generalized Kohlberg-Mertens structure theorem on the
Nash equilibria shown by Predtetchinski \cite{predtetchinski-2009}
holds for semidefinite games (for related recent structural
results in the context of polytopal games see \cite{pahl-2023}).
Moreover, generically,
the number of Nash equilibria in semidefinite games is finite and 
odd, as a consequence of the results of Bich and 
Fixary~\cite{bich-fixary-2021}.

\section{Semidefinite games with many Nash equilibria\label{se:many-nash}}

We construct a family of semidefinite games on the strategy space
$\sym_n \times \sym_n$ such that the set of Nash equilibria has
many connected components. In particular, the number
of Nash equilibria is larger than the number of Nash equilibria
that an $n \times n$-bimatrix game can have.

The following criterion allows us to construct semidefinite games from 
bimatrix games that contain the Nash equilibria of the bimatrix games
and possibly additional ones.

\begin{lemma}
\label{le:preserve-nash}
Let $G=(A,B)$ be an $m \times n$ bimatrix game. 
Let 
$\bar{G}=(\bar{A},\bar{B})$ be a semidefinite game on the strategy space $\sym_m \times \sym_n$
with $\bar{a}_{iikk} = a_{ik}$ and $\bar{b}_{iikk} = b_{ik} $ for $1 \le i \le m$, $1 \le k \le n$.
If $\bar{a}_{ijkk} =0$ for all $i \neq j$ and all $k$,
as well as
$\bar{b}_{iikl} =0$ for all $i$ and all $k \neq l$
then, for every Nash equilibrium
$(x,y)$ of $G$, the pair $(X,Y)$ defined by
\[
  X_{ij} = \begin{cases}
    x_i, & i = j, \\
    0,   & i \neq j,
  \end{cases}
  \qquad
  Y_{ij} = \begin{cases}
    y_i, & i = j, \\
    0, & i \neq j
  \end{cases}
\]
is a Nash equilibrium of $\bar{G}$.
\end{lemma}

\begin{proof}
Assume $(X,Y)$ is not a Nash equilibrium of $\bar{G}$. W.l.o.g. we can assume that a strategy $Y'$ exists for the second player with 
$p_{\bar{B}}(X,Y)<p_{\bar{B}}(X,Y')$.
This yields
\[
p_{\bar{B}}(X,Y)=\sum_{i,j,k,l} \bar{b}_{ijkl} X_{ij}Y_{kl} = 
\sum_{i,k} \bar{b}_{iikk} X_{ii}Y_{kk}
< \sum_{i,k} \bar{b}_{iikk} X_{ii}Y'_{kk} = p_{\bar{B}}(X,Y'),
\]
where the last equation uses that $\bar{b}_{iikl} = 0$ for all $l \neq k$.
Hence, there exists a feasible strategy $y'$ for the second player in $G$
with $p_B(x,y) < p_B(x,y')$. This contradicts the precondition that $(x,y)$
is a Nash equilibrium in $G$.
\end{proof}

In Lemma~\ref{le:preserve-nash}, besides the Nash equilibria inherited from
the bimatrix game, there can be additional Nash equilibria in the semidefinite
games. In a $2 \times 2$-bimatrix game, there can be at most 
three
isolated Nash equilibria (see, e.g., \cite{bgt-1989} 
or \cite{QuiShy-conj-97}).
The following example provides an instance of a semidefinite 
game on the strategy space $\sym_2 \times \sym_2$
with five isolated Nash equilibria. In particular, it has more isolated Nash equilibria than
a $2 \times 2$-bimatrix game can have.

\begin{example}\label{ex:5nasheq}
For a given $c \in \R$, let
\[
  \left( \begin{array}{cc}
    A_{\cdot\cdot 11} & A_{\cdot\cdot 12} \\
    A_{\cdot\cdot 21} & A_{\cdot\cdot 22}
  \end{array} \right)
  =
  \left( \begin{array}{cc}
  \begin{pmatrix} 1 & 0 \\ 0 & 0 \end{pmatrix}
  & \begin{pmatrix} 0 & c \\ c & 0 \end{pmatrix} \\ [2ex]
  \begin{pmatrix} 0 & c \\ c & 0 \end{pmatrix}
  & \begin{pmatrix} 0 & 0 \\ 0 & 1 \end{pmatrix} 
 \end{array} \right)
\]
and $B=A$.
We claim that for $c > 1/2$, there are exactly five isolated Nash equilibria.

First consider the case that the diagonal of $X$ is $(1,0)$. Since $X \succeq 0$, this
implies $x_{12} = 0$. The best response of player~2 gives on the diagonal $(1,0)$ of $Y$.
From that, we see that
\[
  X = Y = \begin{pmatrix} 1 & 0 \\ 0 & 0 \end{pmatrix}
\]
is a Nash equilibrium with payoff 1 for both players, and similarly,
\[
  X = Y = \begin{pmatrix} 0 & 0 \\ 0 & 1 \end{pmatrix}
\]
as well. These Nash equilibria are isolated, which follows as a direct
consequence of the current case in connection with the subsequent
considerations of the cases with diagonal of $X$ not equal to $(1,0)$.
 
Now consider the situation that both diagonal entries of $X$ are positive,
and due to the situation discussed before, we can also assume that both
diagonal entries of $Y$ are positive.
Note that the payoff of each player is
\[
  p(X,Y) = x_{11} y_{11} + x_{22} y_{22} + 4c x_{12} y_{12}.
\]
In a Nash equilibrium, as soon as one player plays the non-diagonal
entry with non-zero weight, then both players will play the non-diagonal
element with maximal possible absolute value and appropriate sign, say,
for player 1, $x_{12} = \pm \sqrt{x_{11} x_{22}}$.

\smallskip

\emph{Case 1:} $x_{12} \neq 0$. We can assume
positive signs for the non-diagonal elements of both players. The payoffs are
\[
  p(X,Y) = x_{11} y_{11} + x_{22} y_{22} 
    + 4c \sqrt{x_{11} x_{22}} \sqrt{y_{11} y_{22}}.
\]
Expressing $x_{22} = 1-x_{11}$ and $y_{22} = 1 - y_{11}$, we obtain
\[
  p(X,Y) = x_{11} y_{11}+ (1-x_{11}) (1-y_{11}) 
    + 4c \sqrt{x_{11} (1-x_{11})} \sqrt{y_{11} (1-y_{11})}.
\]
In a Nash equilibrium, the partial derivatives
\begin{align*}
  p_{x_{11}} = & \ 2 y_{11} - 1 + 
  \frac{2 c \sqrt{y_{11} (1-y_{11})} (1-2 x_{11})}{\sqrt{x_{11} (1-x_{11})}}, \\
  p_{y_{11}} = \ & 2 x_{11} - 1 + 
  \frac{2 c \sqrt{x_{11} (1-x_{11})} (1-2 y_{11})}{\sqrt{y_{11} (1-y_{11})}}
\end{align*}
of $p(X,Y)$ necessarily must vanish. We remark that, since the payoff 
function is bilinear, non-infinitesimal deviations are not relevant here.\\
For the case $c > 1/2$, we obtain 
$x_{11} = y_{11} = 1/2$. For $c = 1/2$, any choice of $x_{11} \in (0,1)$ and setting $y_{11} = x_{11}$ gives a
critical point, see below.

For $x_{11} = \frac{1}{2}$ and $y_{11} = \frac{1}{2}$, we obtain the
Nash equilibria
\[
  X=Y= 
  \left( \begin{array}{cc}
  1/2 & 1/2 \\
  1/2 & 1/2
   \end{array} \right)
   \; \text{ as well as } \;
  X=Y= 
  \left( \begin{array}{cc}
  1/2 & -1/2 \\
  -1/2 & 1/2
  \end{array} \right)
\]
with payoff $\frac{1}{2} + 4 \cdot \frac{1}{4} \cdot c = \frac{1}{2} + c$ 
for both players.

\smallskip

\emph{Case 2:} $x_{12} = 0$. This implies $y_{12} = 0$, and we obtain
the isolated Nash equilibrium
\[
  X = Y =  \left( \begin{array}{cc}
  1/2 & 0 \\
  0 & 1/2
  \end{array} \right)
\]
with payoff $\frac{1}{2}$ for both players.

In the special case $c=1/2$, any choice for $x_{11} \in (0,1)$ and setting
$y_{11} = x_{11}$ gives a critical point. Further inspecting the 
second derivatives
\[
    p_{x_{11} x_{11}} \big|_{y_{11} = x_{11}} \ = \
       - \frac{1}{2 x_{11}(1-x_{11})}  \: \text{ and } \:
    p_{y_{11} y_{11}} \big|_{y_{11} = x_{11}} \ = \
       - \frac{1}{2 x_{11}(1-x_{11})} \, ,
\]
the negative values show that the points are all local maxima w.r.t. deviating
from $x_{11}$ (and, analogously, from $x_{22}$). Hence,
in case $c=1/2$, all the points with $x=y$ for $x \in (0,1)$ and choosing the
maximal possible off-diagonal entries (with appropriate sign with respect
to the other player) give a family of Nash equilibria with payoff~1 for
each player.
\end{example}

\begin{example}\label{ex:5nasheqb}
We can use the Example~\ref{ex:5nasheq} also to illustrate a situation,
where there exists more than one decompositions of a strategy 
into rank-1 matrices. Consider again the main situation $c > \frac{1}{2}$.
We have seen that the pair
$(\frac{1}{2} I_2, \frac{1}{2} I_2)$ of scaled identity matrices
constitutes a Nash equilibrium. If player~1 plays 
$X=\frac{1}{2} I_2$, the payoff of player~2 is
\[
  p_B(X,Y) = \frac{1}{2} y_{11} + \frac{1}{2} y_{22} + 0 \cdot c \, ,
\]
which is independent of $c$. Due to $\tr(Y) = 1$, the payoff 
is~$\frac{1}{2}$ 
for any strategy $Y$ of the second player.
Note that the unit matrix has several decompositions into rank 1-matrices.
Besides the canonical decomposition
\[
  I_2 \ = \ e^{(1)} (e^{(1)})^T + e^{(2)} (e^{(2)})^T,
\]
we can also consider, say, even for a general unit matrix $I_n$, 
the decomposition
\[
  I_n \ = \ \sum_{k=1}^n (u^{(k)}) (u^{(k)})^T
\]
for any orthonormal basis $u^{(1)}, \ldots, u^{(n)}$ of $\R^n$.
Both for the canonical decomposition and for the decomposition, say,
with $u^{(1)} = (\cos \alpha, \sin \alpha)^T$, 
$u^{(2)} = (-\sin \alpha,  \cos \alpha)^T$ for $\alpha:=\frac{\pi}{6}$, 
i.e., $u^{(1)} = \frac{1}{2}(\sqrt{3},1)^T$,
$u^{(2)} = \frac{1}{2}(-1,\sqrt{3})^T$,
we obtain $p_B(X,Y)=\frac{1}{2}$. In particular, all the rank 1-strategies
occurring in the various decompositions of $\frac{1}{2}I_2$ are best responses
of player~2 to the strategy $\frac{1}{2} I_2$ of the first player.
\end{example}

We now show how to construct from Example~\ref{ex:5nasheq} 
an explicit family of semidefinite games with many Nash equilibria.

\smallskip

\noindent
\emph{Block construction.}
Let $A^{(1)}$ and $A^{(2)}$ be tensors of size
$m_1 \times m_1 \times n_1 \times n_1$ and
$m_2 \times m_2 \times n_1 \times n_1$.
The \emph{block tensor with blocks} $A^{(1)}$
and $A^{(2)}$ is formally defined as the tensor of size
\[
  (m_1 + m_2) \times
  (m_1 + m_2) \times
  (n_1 + n_2) \times
  (n_1 + n_2),
\]
which has entries
\[
  \begin{array}{rclll}
  a_{ijkl} & = & a^{(1)}_{ijkl} & \text{ for all } & 1 \le i,j \le m_1, \,
    1 \le k,l \le n_1, \\
 a_{i+m_1, j+m_1,k+n_1,l+n_1}
 & = & a^{(2)}_{ijkl} & \text{ for all } & 1 \le i,j \le m_2, \,
    1 \le k,l \le n_2.
  \end{array}
\]
Naturally, this construction can be extended to more than two blocks.

For $1 \le k \le 2$, let $(A^{(k)}, B^{(k)})$ be a semidefinite game
$G^{(k)}$ with strategy space $\sym_{m_k} \times \sym_{n_k}$.
Then the \emph{block game} $G=(A,B)$, where $A$ is the
block tensor with blocks $A^{(1)}$ and $A^{(2)}$ and
$B$ is the block tensor with blocks $B^{(1)}$ and $B^{(2)}$, defines
a semidefinite game with strategy space
$\sym_{m_1 + m_2} \times
 \sym_{n_1 + n_2}$.

\begin{lemma}[Block lemma]
\label{le:block-constr}
For $1 \le k \le 2$, let $G^{(k)}=(A^{(k)}, B^{(k)})$ be a semidefinite game
with strategy space $\sym_{m_k} \times \sym_{n_k}$
and $(X^{(k)}, Y^{(k)})$ be a Nash equilibrium of $G^{(k)}$
and denote the payoffs of the two players by $p_{A^{(k)}}$
and $p_{B^{(k)}}$. If $\alpha_1,\alpha_2,\beta_1,\beta_2 \ge 0$
satisfy
$\alpha_1 + \alpha_2 = 1, \beta_1 + \beta_2 = 1$ as well as
\[
  \alpha_1 p_{B^{(1)}}(X^{(1)},Y^{(1)}) 
  = \alpha_2 p_{B^{(2)}}(X^{(2)},Y^{(2)}) \text{ and }
  \beta_1 p_{A^{(1)}}(X^{(1)},Y^{(1)}) 
  = \beta_2 p_{A^{(2)}}(X^{(2)},Y^{(2)}),
\]
then
\[
  X^* := \begin{pmatrix}
    \alpha_1 X^{(1)} & 0 \\
    0 & \alpha_2 X^{(2)}
  \end{pmatrix}, \quad
  Y^* := \begin{pmatrix}
    \beta_1 Y^{(1)} & 0 \\
    0 & \beta_2 Y^{(2)}
  \end{pmatrix}
\]
is a Nash equilibrium of the block game of
$(A^{(1)}, B^{(1)})$ and $(A^{(2)}, B^{(2)})$.
\end{lemma}

Note that if one of the coefficients $\alpha_1,\alpha_2,\beta_1$ or 
$\beta_2$ is zero in the theorem, then one of the blocks in $X^*$
or $Y^*$ consists solely of zeroes.

\begin{proof} We denote the payoff tensors of the block game
by $A$ and $B$. Let the first player play 
$X^*$. Since $\alpha_1 + \alpha_2 = 1$,
$X^*$ is indeed an admissible strategy of the first player.
If the second player plays a strategy $\bar{Y}$,
we can assume that it is of the form
\[
  \bar{Y} = \begin{pmatrix}
    \gamma_1 \bar{Y}^{(1)} & 0 \\
    0 & \gamma_2 \bar{Y}^{(2)}
  \end{pmatrix}
\]
with some $\gamma_1, \gamma_2 \ge 0$,
$\gamma_1 + \gamma_2 = 1$ 
and strategies $\bar{Y}^{(1)}$, $\bar{Y}^{(2)}$ of
$G^{(1)}$ and $G^{(2)}$.
Since $(X^{(k)}, Y^{(k)})$ is a Nash equilibrium of $G^{(k)}$ for
$k \in \{1,2\}$, we obtain
\begin{eqnarray*}
  p_B(X^*,\bar{Y}) & = & \alpha_1 \gamma_1
  p_{B^{(1)}}(X^{(1)},\bar{Y}^{(1)})
  + \alpha_2 \gamma_2 p_{B^{(2)}}(X^{(2)}, \bar{Y}^{(2)}) \\
  & \le & \alpha_1 \gamma_1 p_{B^{(1)}}(X^{(1)},Y^{(1)})
  + \alpha_2 \gamma_2 p_{B^{(2)}}(X^{(2)}, Y^{(2)}) \\
  & = & (\gamma_1 + \gamma_2) \alpha_1 p_{B^{(1)}}(X^{(1)}, Y^{(1)}) \\
  & = & (\beta_1 +\beta_2) \alpha_1 p_{B^{(1)}}(X^{(1)}, Y^{(1)}) \\
  & = &  \alpha_1 \beta_1 p_{B^{(1)}}(X^{(1)}, Y^{(1)}) 
        + \alpha_2 \beta_2 p_{B^{(2)}}(X^{(2)}, Y^{(2)}) \\
  & = & p_B(X^*,Y^*) \, .
\end{eqnarray*}
An analogous argument holds for the best response of the first player
to  the strategy $Y^*$ of the second player.
\end{proof}

We can use the block lemma to construct a family of semidefinite
games with many Nash equilibria.

\begin{example}
\label{ex:many-nash1}
Let $m=n$, i.e., we consider a game $G_n$ on 
$\sym_n \times \sym_n$. Assume that $n$ is even.
We generalize 
Example~\ref{ex:5nasheq}. For a given $c \in \R$, let
\[
  \left( \begin{array}{cc}
    A_{\cdot\cdot 11} & A_{\cdot\cdot 12} \\
    A_{\cdot\cdot 21} & A_{\cdot\cdot 22}
  \end{array} \right)
  =
  \left( \begin{array}{cc}
  \begin{pmatrix} 1 & 0 \\ 0 & 0 \end{pmatrix}
  & \begin{pmatrix} 0 & c \\ c & 0 \end{pmatrix} \\ [2ex]
  \begin{pmatrix} 0 & c \\ c & 0 \end{pmatrix}
  & \begin{pmatrix} 0 & 0 \\ 0 & 1 \end{pmatrix} 
 \end{array} \right)
\]
as in Example~\ref{ex:5nasheq}.
For $1 < s < n/2$ and $2s-1 \le \{i,j,k,l\} \le 2s$, let
\[
  A_{ijkl} \ = \ A_{i-2(s-1),j-2(s-1),k-2(s-1),l-2(s-1)}
\]
and let all other entries of $A$ be zero. Also, let $B=A$.

The discussion in Example~\ref{ex:5nasheq}, for the specific situation of the
game on the strategy sets $\sym_2 \times \sym_2$, 
implies that there are five Nash equilibria. In all these equilibria, the strategies
of both players coincide and this property is preserved 
throughout the generalized construction we present.
\end{example}

\begin{thm}
\label{th:nashnumber}
The set of Nash equilibria of $G_n$ consists of
\[
  \sqrt{6}^n - 1 \ \approx \  2.449^n - 1
\]
connected components.
\end{thm}

\begin{proof}
Using the Block Lemma~\ref{le:block-constr}, we obtain
$
  6^{n/2} - 1 = \sqrt{6}^n -1
$
Nash equilibria, because we can also use the zero matrix as 
$2 \times 2$ block within a strategy
as long as not all the blocks are the zero matrix. Outside 
of the $2 \times 2$-diagonal blocks, the entries of these Nash equilibria
are zero. Those entries can be chosen arbitrarily as long as
the positive semidefiniteness constraint on the strategy is satisfied,
without losing the equilibrium property. As a consequence,
the Nash equilibria are not isolated.
It remains to show that the $\sqrt{6}^n -1$ Nash equilibria 
obtained from the Block Lemma belong to distinct connected 
components.

For each Nash equilibrium $(X,Y)$, consider the diagonal 
$2 \times 2$-blocks of the strategies. In each block, we have one of 
the five types from Example~\ref{ex:5nasheq} or the block 
is the zero matrix.
We associate a type 
$p(X,Y) \in \{0, \ldots, 5\}^{n}$ to
each Nash equilibria which gives the type in each of the $n/2$
blocks of $X$ and in each of the $n/2$ blocks of $Y$.

Any two of the $\sqrt{6}^n-1$ Nash equilibria coming from the
Block Lemma have distinct types. By restricting to the diagonal
blocks, this implies that the $6^{n/2}-1$ Nash equilibria
belong to distinct connected components.
\end{proof}

Asymptotically, we obtain more Nash equilibria than in the Quint
and Shubik construction of bimatrix games \cite{QuiShy-conj-97}
and also more Nash equilibria than in
von Stengel's construction of bimatrix games \cite{von-stengel-1999},
because there the number is $0.949 \cdot 2.414^n/\sqrt{n}$ asymptotically.

Specifically, von Stengel's construction gives a 
$6 \times 6$-bimatrix game with 75 isolated 
Nash equilibria, and so far no $6 \times 6$-bimatrix game with more 
than 75 isolated Nash equilibria is known. Von Stengel also showed an
upper bound of 111 Nash equilibria for a $6 \times 6$-bimatrix games.
In our construction of a semidefinite game on $\sym_6 \times \sym_6$, 
we obtain from Theorem~\ref{th:nashnumber}
the higher number of
$6^{6/2} - 1 = 215$ connected components of Nash equilibria
in the semidefinite game.

\section{Outlook and open questions\label{se:outlook}}

Since the transition from bimatrix games to semidefinite games
leads from polyhedra to spectrahedra, in the geometric description
of Nash equilibria, the questions on the maximal number of Nash
equilibria appear to become even more challenging than in 
the bimatrix situation. Both from the viewpoint of the combinatorics
of Nash equilibria and from the viewpoint of computation, rank
restrictions have been fruitfully exploited in the contexts of bimatrix
games \cite{agm-2021,kannan-theobald-2010} and 
separable games \cite{sop-2008}. It would be interesting to
study the exploitation 
of low-rank structures
of semidefinite games in the case of payoff tensors with suitable
conditions on a low tensor rank.

Concerning the reduction from semidefinite programs to
semidefinite games, it is a natural question whether the
handlings of the exceptional cases by Adler \cite{adler-2013} and 
von Stengel \cite{von-stengel-2022} can be generalized
to the semidefinite case.

We also briefly mention questions of the semidefinite generalizations
of more general classes of (bimatrix) games.
A \emph{polymatrix game} (or \emph{network game})~\cite{ccdp-polymatrix-16} is defined by a graph.
The nodes are the players and each edge corresponds to a two-player
zero-sum matrix game. Every player chooses one set of strategies 
and 
she uses it with all the games that she is involved with.
The game has an equilibrium that we can compute efficiently using 
linear programming.
Shapley's \emph{stochastic games} are 
 two-player zero-sum games of potentially infinite duration.
 Roughly speaking, the game takes place on a complete graph, 
 the nodes of which correspond 
 to zero-sum matrix games. 
 Two players, starting from an arbitrary node (position), 
at each stage of the game, play a 
zero-sum matrix game and receive payoffs. Then, with a non-zero probability the game either stops 
or they players move to another node and play again. Because the stopping probabilities are non-zero at each position, the game terminates. Shapley proved \cite{shapley_stochastic_1953} that this game has an equilibrium; there is also an algorithm to compute it \cite{hansen2011exact}, see also \cite{oliu2021new}.
It remains a future task to study the generalizations of polymatrix
 and stochastic games when the underlying 
bimatrix games are replaced by semidefinite games. For 
semidefinite polymatrix games this has recently been initiated 
in~\cite{itt-2023}.

\subsubsection*{Acknowledgements}
The authors are thankful to 
Giorgos Christodoulou and Antonios Varvitsiotis
as well as to the anonymous referees
for their comments and suggestions.

\bibliographystyle{abbrv}

\bibliography{bib-semidefinite-games}

\end{document}